\documentclass[11pt,a4paper]{article}

\usepackage{psfrag}
\usepackage{amsmath,amssymb,cite}
\usepackage{latexsym}
\usepackage{graphicx}
\usepackage{lscape}
\usepackage{afterpage}

\DeclareMathOperator*{\argmin}{argmin}
\DeclareMathOperator*{\sgn}{sgn}

\newcommand{\ds}{\displaystyle}
\newcommand{\nexto}{\kern -0.54em}
\newcommand{\dR}{{\rm {I\ \nexto R}}}

\newcommand{\dZ}{{\cal Z \kern -0.7em Z}}
\newcommand{\dC}{{\rm\hbox{C \kern-0.8em\raise0.2ex\hbox{\vrule
height5.4pt width0.7pt}}}}
\newcommand{\dQ}{{\rm\hbox{Q \kern-0.85em\raise0.25ex\hbox{\vrule
height5.4pt width0.7pt}}}}
\newcommand{\proofbox}{\hspace{\fill}{$\Box$}}
\newtheorem{lemma}{Lemma}
\newtheorem{theorem}{Theorem}
\newtheorem{corollary}{Corollary}

\newtheorem{remark}{Remark}

\newenvironment{proof}{Proof.}{\proofbox}

\begin{document}

\author{Authors}

\author{
C. Yal{\c c}{\i}n Kaya\footnote{School of Information Technology and Mathematical Sciences, University of South Australia, Mawson Lakes, S.A. 5095, Australia. E-mail: yalcin.kaya@unisa.edu.au\,.}
}

\title{\vspace{0mm}\bf Markov-Dubins Path via Optimal Control Theory}

\maketitle

\begin{abstract} 
{\noindent\sf Markov-Dubins path is the shortest planar curve joining two points with prescribed tangents, with a specified bound on its curvature.  Its structure, as proved by Dubins in 1957, nearly 70 years after Markov posed the problem of finding it, is elegantly simple: a selection of at most three arcs are concatenated, each of which is either a circular arc of maximum (prescribed) curvature or a straight line.  The Markov-Dubins problem and its variants have since been extensively studied in practical and theoretical settings.  A reformulation of the Markov-Dubins problem as an optimal control problem was subsequently studied by various researchers using the Pontryagin maximum principle and additional techniques, to reproduce Dubins' result.  In the present paper, we study the same reformulation, and apply the maximum principle, with new insights, to derive Dubins' result again.  We prove that abnormal control solutions do exist. We characterize these solutions, which were not studied adequately in the literature previously, as a concatenation of at most two circular arcs and show that they are also solutions of the normal problem.  Moreover, we prove that any feasible path of the types mentioned in Dubins' result is a stationary solution, i.e., that it satisfies the Pontryagin maximum principle.  We propose a numerical method for computing Markov-Dubins path.  We illustrate the theory and the numerical approach by three qualitatively different examples.}
\end{abstract}

\begin{verse} {\bf Key words.} {\sf Markov-Dubins path, Bounded curvature, Optimal control, Singular control, Bang--bang control, Abnormal optimal control problem.}
\end{verse}

\begin{verse} {\bf AMS subject classifications.} {\sf Primary 49J15, 49K15\ \ Secondary 65K10, 90C30}
\end{verse}

\pagestyle{myheadings}
\thispagestyle{plain}
\markboth{\sf\scriptsize C. Y. Kaya}{\sf\scriptsize Markov-Dubins Path via Optimal Control Theory\ \ by\ C. Y. Kaya}

\section{Introduction}

The problem of finding the shortest planar path of bounded curvature between two prescribed end-points and tangents in the plane was posed in 1889 by Andrey Andreyevich Markov~\cite{Markov1889}, in the context of railway design (also see \cite{KreNud1977}).  Back then, Markov studied only some of the specific instances of this problem.   Nearly 70 years later (in 1957), Lester Eli Dubins~\cite{Dubins1957} published a general solution to the problem for the first time, using geometric arguments.  {\em Dubins path}, as coined by many authors in the literature but which we refer to here as {\em Markov-Dubins path} in recognition of Markov's earlier contribution in~\cite{Markov1889}, is the shortest $C^1$ and piecewise-$C^2$ curve that is a concatenation of circular subarcs and a straight line.  Suppose that a circular arc is represented by $C$ and a straight line segment by $S$.  Dubins' elegant solution asserts that the sequence of concatenated arcs in such a shortest path can be of type $CSC$, $CCC$, or a subset thereof.  

Markov-Dubins path and its variants have since been extensively studied for optimal path planning of uninhabited aerial vehicles (UAVs)~\cite{GalDeu2014, GaoZheGon2016,WanWanTanZhoWei2015} and robots~\cite{TokKarIsl2014}. They have also been used for tunnelling in underground mines, where it is paramount to minimize the cost of excavating and operating a tunnel \cite{ChaBraRubTho2012, ChaBraRubTho2015}.   Markov-Dubins path and its generalizations have been an active area of research for decades now~\cite{AgrSac2004,AyaRub2016, BakTsi2013, IsaShi2014, Sussmann1995, Sussmann1997, SusTan1991, FraSch2004, ChiSig2005, MeyIsaShi2015, SigChi2006,  BoiCerLeb1991, BoiCerLeb1994, ReeShe1990, ShkLum2001}.

A straightforward exemplification of Markov-Dubins path is the shortest path of a car (modelled as a point mass) which goes only forwards at unit speed under the same constraints described above.  In 1990, Reeds and Shepp~\cite{ReeShe1990} considered a car which goes not only forwards but also backwards, as an extension of the Markov-Dubins problem.  This extension clearly allows reversals of the path, which constitute cusps.  Therefore, the shortest curve for the Reeds-Shepp car is no longer necessarily $C^1$, and the results are somewhat richer.  In 1991, Boissonnat, C{\'e}r{\'e}zo and Leblond~\cite{BoiCerLeb1991} (also see~\cite{BoiCerLeb1994}) used optimal control theory, as well as perturbation analysis of the solution trajectories, to derive the same results as those obtained by Reeds and Shepp.  Sussmann and Tang~\cite{SusTan1991} independently did the same in 1991, by using geometric optimal control theory and control synthesis.

In both \cite[Lemma~5]{BoiCerLeb1991} and \cite[Lemma~2]{SusTan1991}, the optimal control problem associated with the Reeds-Shepp car is proved to be {\em normal}, i.e., the multiplier of the objective functional is positive (non-zero). In \cite{BoiCerLeb1991}, however, normality is carried over to the case of the forward-moving car as a special case of the Reeds-Shepp car, and the results that are derived for the Reeds-Shepp car are reduced to those of Dubins.  Under the assumption that the optimal control problem is normal, the optimal curve types one gets indeed overlap with those of the result of Dubins;  however, it turns out that the problem may very well be {\em abnormal}, i.e., the multiplier of the objective functional may become zero.  Abnormality of the Markov-Dubins problem is observed in \cite[Remark~11]{SusTan1991}; however, abnormal solutions have not been studied explicitly up to now.  Therefore, there is need to characterize abnormal solutions adequately for completeness of the optimal control approach.

In the present paper, we formulate the Markov-Dubins problem as a time-optimal control problem as in \cite{BoiCerLeb1991} and \cite{SusTan1991} and apply the Pontryagin maximum principle.  We prove that there exist abnormal, as well as normal, solutions to the Markov-Dubins problem, and we characterize these.  We carry out perturbation analysis of the solution trajectories, for the normal and abnormal cases, so as to get the results reported by Dubins.

First, as in \cite{BoiCerLeb1991} and \cite{SusTan1991}, we show that the optimal control trajectories are comprised of bang and singular arcs such that a bang arc is associated with a circular arc $C$ and a singular arc with a straight line segment $S$.  Then we derive a differential equation for the switching function and construct the phase portraits of this differential equation for the normal and abnormal cases (see Figures~\ref{phase} and \ref{abnormal_phase}).  These phase portraits exhibit markedly different phase plane trajectories from one another, which help with the characterization of the possible types of solutions.  

From a simple lemma (Lemma~\ref{rho_sing}), as well as the abnormal phase portrait (Figure~\ref{abnormal_phase}), it becomes evident that an abnormal solution curve cannot contain a singular subarc. Next, by using perturbation analysis, we show that an abnormal optimal solution curve must be either of type $C$ or $CC$ (Lemma~\ref{abnormal}).  By using a similar perturbation analysis, we also show that a normal optimal solution curve cannot be of type $CCCC$ (Lemma~\ref{CCCC}).  The latter analysis is akin to that in \cite{BoiCerLeb1991}; however, the tools and details of our working are different.  A combination of all these results reproduce the earlier result given by Dubins (Theorem~\ref{Dubins}).  In other words, by using optimal control theory and perturbation analysis, we provide a (full) alternative proof of Dubins' result.

In another new result (Theorem~\ref{stationarity}), we state that any feasible path, i.e., any path satisfying the constraints of the optimal control formulation of the Markov-Dubins problem, which is of type $CSC$, $CCC$, or a subset thereof, is a {\em stationary} solution, i.e., it satisfies the Pontryagin maximum principle.  Stationarity of feasible solutions has computational implications as pointed in the next paragraph.  The result in Corollary~\ref{normal_abnormal_1} illustrates that abnormal Markov-Dubins path is not rare at all.  Corollary~\ref{normal_abnormal_2}, on the other hand, states that any abnormal path is a normal path, but not vice versa.

Since the solution structure for the optimal control problem is a concatenation of bang and singular arcs, we parameterize the problem with respect to the unknown terminal time and the switching times at which the solution curve switches from one subarc to the other.  We propose a model for switching-time computation in a similar fashion to those in~\cite{KayNoa2003,MauBueKimKay2005}, which converts the Markov-Dubins problem, which is an infinite-dimensional optimization problem, to an equivalent finite-dimensional one.  The transformed problem can then be solved using standard optimization methods and software.  It is well-known that most iterative methods for optimization would in general converge to a stationary solution, but not necessarily to a locally optimal one, let alone a globally optimal one.  So, by using standard optimization software, one would hope to get at best, by virtue of Theorem~\ref{stationarity}, one of the usually many feasible solutions of the Markov-Dubins problem.

Three example Markov-Dubins problems are studied.  In the first two, we construct all of the (normal) stationary solutions.  For the Markov-Dubins path of Example~1, we illustrate the construction of the switching function and the phase plane trajectories.  In Example~2, we draw six of all seven stationary curves.  Example~3 illustrates four stationary curves, including the optimal one, the Markov-Dubins path, which is abnormal.

A realistic generalization of the Markov-Dubins problem is the requirement that the path passes through a number of prescribed intermediate points, giving rise to an interpolation problem.  Although this generalization is beyond the scope of the present paper, the material is presented in such a way that it can be conveniently/simply extended to study what we call {\em the Markov-Dubins interpolation problem}.  The particular numerical approach we propose in this paper constitutes a crucial building block of a numerical method for this extension.

The paper is organized as follows.  In Section~2, we describe the Markov-Dubins problem and its reformulation as a time-optimal control problem.  In Section~3, we provide the preliminary results, including the abnormal and normal solutions, leading to Dubins' theorem.  The stationarity results are presented in Section~4.  In Section~5, we describe the numerical approach and present the numerical examples and experiments.  Section~6 concludes the paper, with a discussion and short descriptions of further work.

\section{Markov-Dubins Problem}

{\em Markov-Dubins path} is the shortest $C^1$ and piecewise-$C^2$ curve $z:[0,t_f]\longrightarrow\dR^2$ between two prescribed oriented points $p_0$ and $p_f$ at $0$ and $t_f$, respectively, in the plane, where the slopes at $p_0$ and $p_f$ are also prescribed, such that the curvature of the path $z(t)$ at almost every (a.e.) point is not greater than $a>0$.  Note that the parameter $t_f$ is unknown, and so it is also to be determined.

Recall that the {\em curvature} $\kappa(t)$ of a $C^2$ curve, parameterized with respect to its length, is defined as the length of its acceleration vector; in other words,
\[
\kappa(t) = \|\ddot{z}(t)\|\,, \quad\mbox{with\ } \|\dot{z}(t)\| = 1\,,
\]
where $\dot{z} = dz/dt$, $\ddot{z} = d^2z/dt^2$, and $\|\cdot\|$ is the Euclidean norm.  So, it is required that $\kappa(t) \le a$, for a.e.\ $t\in[0,t_f]$.  In practical terms, if we consider, for example, the motion of a vehicle in the plane, $1/\kappa(t)$ is nothing but the (instantaneous) {\em turning radius}.  So, in the Markov-Dubins problem, the turning radius is constrained to be at least $1/a$.  Finding the shortest curve requires minimization of the arc-length functional
\[
\int_{0}^{t_f}\|\dot{z}(t)\|\,dt = t_f\,,
\]
where we have used $\|\dot{z}(t)\| = 1$, for a.e.\ $t\in[0,t_f]$.  Note that this reconfirms parameterization of the curve with respect to its length.  The Markov-Dubins problem can then be expressed as
\[
\mbox{(P)}\left\{\begin{array}{rl}
\min &\ t_f \\[2mm]
\mbox{s.t.} &\ z(0) = p_0,\ z(t_f) = p_f,\\[2mm]
  &\ \dot{z}(0) = v_0,\ \dot{z}(t_f) = v_f,\\[2mm]
   &\ \|\ddot{z}(t)\|\le a\,,\ \ \|\dot{z}(t)\| = 1\,,\mbox{ for a.e. }  t\in[0,t_f]\,,
\end{array}\right.
\]
where $\|v_0\| = \|v_f\| = 1$.  Problem~(P) can equivalently be cast as an optimal control problem as follows. Let $z(t) := (x(t), y(t))\in\dR^2$, with $\dot{x}(t) := \cos\theta(t)$ and $\dot{y}(t) := \sin\theta(t)$, where $\theta(t)$ is the angle the velocity vector $\dot{z}(t)$ of the curve $z(t)$ makes with the horizontal.  These definitions verify that $\|\dot{z}(t)\| = 1$.  Moreover,
\[
\|\ddot{z}\|^2 = \ddot{x}^2+\ddot{y}^2 = \dot{\theta}^2\,.
\]
Therefore, $|\dot{\theta}(t)|$ is nothing but the curvature.  In fact, $\dot{\theta}(t)$ itself, which can be positive or negative, is referred to as the {\em signed curvature}.  For example, consider a vehicle travelling along a circular path.  If $\dot{\theta}(t) > 0$ then the vehicle travels in the counter-clockwise direction, i.e., it {\em turns left}, and if $\dot{\theta}(t) < 0$ then the vehicle travels in the clockwise direction, i.e., it {\em turns right}.

Let $u(t) := \dot\theta(t)$.  Suppose that the directions at the points $p_0$ and $p_f$ are denoted by the angles $\theta_0$ and $\theta_f$, respectively.  Problem~(P) can then be re-written as a {\em time-optimal} (or {\em minimum-time}) {\em control problem}, where $x$, $y$ and $\theta$ are the {\em state variables} and $u$ the {\em control variable}:
\[
\mbox{(Pc)}\left\{\begin{array}{rll}
\min &\ \ds t_f = \int_0^{t_f} (1)\, dt  & \\[4mm]
\mbox{s.t.} &\ \dot{x}(t) = \cos\theta(t)\,, & x(0) = x_0\,,\ 
              x(t_f) = x_f\,, \\[2mm] 
  &\ \dot{y}(t) = \sin\theta(t)\,, & y(0) = y_0\,,\ 
              y(t_f) = y_f\,,\\[2mm] 
  &\ \dot{\theta}(t) = u(t)\,, & \theta(0) = \theta_0\,,\ 
              \theta(t_f) = \theta_f\,,\\[2mm]
  & & |u(t)|\le a\,,\mbox{ for a.e. }  t\in[0,t_f]\,.
\end{array}\right.
\]

\section{Markov-Dubins Curves}

In this section, we apply the Pontryagin maximum principle to Problem~(Pc) and ultimately reproduce the result obtained by Dubins~\cite{Dubins1957}.

\subsection{Bang--bang and singular arcs}

Define the {\em Hamiltonian function} as in \cite{PonBolGamMis1962} for Problem~(Pc) as:
\begin{equation}  \label{Hamiltonian}
H(x,y,\theta, \lambda_0,\lambda_1,\lambda_2,\lambda_3,u) := \lambda_0
+ \lambda_1\,\cos\theta + \lambda_2\,\sin\theta + \lambda_3\,u\,,
\end{equation}
where $\lambda_0$ is a scalar (multiplier) parameter and $\lambda_i: [0,t_f]\to\dR$, $i = 1,2,3$, are the adjoint (or costate) variables.  Let
\[
H[t] := H(x(t),y(t),\theta(t),\lambda_0,\lambda_1(t),\lambda_2(t),\lambda_3(t),u(t))\,. 
\]
The adjoint variables are required to satisfy
\begin{eqnarray}
&& \dot{\lambda}_1(t) = -H_x[t] = 0\,, \label{adjoint1} \\[1mm]
&& \dot{\lambda}_2(t) = -H_y[t] = 0\,, \label{adjoint2} \\[1mm]
&& \dot{\lambda}_3(t) = -H_\theta[t] = \lambda_1(t)\,\sin\theta(t)
- \lambda_2(t)\,\cos\theta(t)\,, \label{adjoint3}
\end{eqnarray}
where $H_x = \partial H / \partial x$, etc.  By these definitions, the state and adjoint variables verify a Hamiltonian system in that, in addition to \eqref{adjoint1}--\eqref{adjoint3}, one has $\dot{x}(t) = H_{\lambda_1}[t]$, $\dot{y}(t) = H_{\lambda_2}[t]$ and $\dot{\theta}(t) = H_{\lambda_3}[t]$.  Note that
\eqref{adjoint1}--\eqref{adjoint2} imply that $\lambda_1(t) = \overline{\lambda}_1$ and $\lambda_2(t) = \overline{\lambda}_2$ for all $t\in[0,t_f]$, where $\overline{\lambda}_1$ and $\overline{\lambda}_2$ are constants.

Define new constants
\begin{equation}  \label{rhophi}
\rho := \sqrt{\overline{\lambda}_1^2 +
  \overline{\lambda}_2^2}\,,\qquad 
\tan\phi := \frac{\overline{\lambda}_2}{\overline{\lambda}_1}\,.
\end{equation}
Then \eqref{Hamiltonian} and \eqref{adjoint3} can respectively be re-written as
\begin{equation}  \label{Hamiltonian2}
H[t] = \lambda_0 + \rho\,\cos(\theta(t) - \phi) + \lambda_3(t)\,u(t)
\end{equation}
and 
\begin{equation}  \label{adjoint3a}
\dot{\lambda}_3(t) = \rho\,\sin(\theta(t) - \phi)\,.
\end{equation}

Next we state the Pontryagin maximum principle \cite[Theorem 1]{PonBolGamMis1962} for our setting as follows.  Suppose that $x,y,\theta\in W^{1,\infty}(0,t_f;\dR)$, $u\in L^\infty(0,t_f;\dR)$, and $t_f\in[0,M)$, where $M$ is large enough so that $t_f<M-\varepsilon$ with $\varepsilon>0$, solve Problem~(Pc).  Then
there exist a number $\lambda_0\ge0$ and functions $\lambda_i\in W^{1,\infty}(0,t_f;\dR)$, $i=1,2,3$, such that $\lambda(t ):= (\lambda_0,\lambda_1(t), \lambda_2(t), \lambda_3(t)) \neq \bf0$, for every $t\in[0,t_f]$, and, in addition to the state differential equations and other constraints given in Problem~(Pc) and the adjoint
differential equations \eqref{adjoint1}--\eqref{adjoint2} and \eqref{adjoint3a}, the following conditions hold:
\begin{eqnarray}
&& u(t)\in\argmin_{\|v\|\le a}
   H(x(t),y(t),\theta(t),\lambda_0,\lambda_1(t),\lambda_2(t),\lambda_3(t),v)\,,  
\label{control} 
  \\[1mm]
  && H[t] = 0\,.  \label{H_zero}
\end{eqnarray}
Using the definition in \eqref{Hamiltonian}, \eqref{control} can more simply be written as
\begin{equation}  \label{control2}
u(t)\in\argmin_{\|v\|\le a} \lambda_3(t)\,v
\end{equation}
which yields the optimal control as
\begin{equation}  \label{control3}
u(t) = \left\{\begin{array}{ll}
\ \ a\,, & \mbox{if}\  \lambda_3(t) < 0\,, \\[3mm]
-a\,, & \mbox{if}\ \lambda_3(t) > 0\,, \\[3mm]
\mbox{undetermined}\,, & \mbox{if}\ \lambda_3(t) = 0\,.
\end{array}\right.
\end{equation}
Furthermore, \eqref{Hamiltonian2} and \eqref{H_zero} give
\begin{equation}  \label{H_zero2}
\lambda_3(t)\,u(t) + \rho\,\cos(\theta(t) - \phi) + \lambda_0 = 0\,.
\end{equation}

The control $u(t)$ to be chosen for the case when $\lambda_3(t)=0$ for a.e.\ $t\in[\zeta_1,\zeta_2]\subset[0,t_f]$ is referred to as {\em singular control}, because \eqref{control2} does not yield any further information.  On the other hand, when $\lambda_3(t)\neq 0$ for a.e.\ $t\in[0,t_f]$, i.e., it is possible to have $\lambda_3(t)=0$ only for isolated values of $t$, the control $u(t)$ is said to be {\em nonsingular}.  It should be noted that, if $\lambda_3(\tau)=0$ only at an isolated point $\tau$, the optimal control at this isolated point can be chosen as $u(\tau) = -a$ or $u(\tau) = a$, conveniently.  If the control $u(t)$ is nonsingular, it will take on either the value $-a$ or $a$, the bounds on the control variable.  In this case, the control $u(t)$ is referred to as {\em bang--bang}.  Since the sign of $\lambda_3(t)$ determines the value of the optimal control $u(t)$, $\lambda_3$ is referred to as the {\em switching function}.

\begin{lemma}[Normality of Singular Control] \label{rho_sing} 
  Suppose that the optimal control $u(t)$ for Problem~{\em (Pc)} is singular over an interval $[\zeta_1,\zeta_2]\subset[0,t_f]$.  Then $\rho = \lambda_0 > 0$.
\end{lemma}
\begin{proof}
Suppose that the optimal control $u(t)$ is singular for a.e.\ $t\in[\zeta_1,\zeta_2]\subset[0,t_f]$.  Then $\lambda_3(t) = \dot\lambda_3(t) = 0$ for a.e.\ $t\in[\zeta_1,\zeta_2]$. From \eqref{adjoint3a}, $\sin(\theta(t) - \phi) = 0$, which implies that $\cos(\theta(t) - \phi) = 1$ or $-1$.  Then, substituting $\lambda_3(t)=0$ and $\cos(\theta(t) - \phi) = -1$ into \eqref{H_zero2}, one gets $\rho = \lambda_0 \ge 0$.  It should be noted that $\cos(\theta(t) - \phi) = 1$ yields $\rho = -\lambda_0 \ge 0$, or $\lambda_0 \le 0$, which is not allowed by the maximum principle, unless $\lambda_0 = 0$.  Suppose that $\rho = \lambda_0 = 0$. Then, since $\lambda_3(t) = 0$, the adjoint variable vector $\lambda(t) = {\bf 0}$, which contradicts the maximum principle.  Therefore $\lambda_0 > 0$.
\end{proof}

\begin{remark} \label{rem:singular_normal} \rm The problems that yield $\lambda_0 = 0$ are referred to as {\em abnormal} in the optimal control theory literature, for which the necessary conditions in \eqref{control}--\eqref{H_zero} are independent of the objective functional $t_f$ and therefore insufficiently informative.  The problems that yield $\lambda_0 > 0$ are referred to as {\em normal}.  Lemma~\ref{rho_sing} above asserts that if the optimal path contains a singular arc, then Problem~(Pc) is normal.
\endproof
\end{remark}

\begin{lemma}[Singularity and Straight Line Segments] \label{singular}
Suppose that the optimal control $u(t)$ for Problem~{\em (Pc)} is singular over an interval $[\zeta_1,\zeta_2]\subset[0,t_f]$.  Then $\theta(t)$ is constant, i.e., $u(t) = 0$.
\end{lemma}
\begin{proof}
Suppose that the optimal control $u(t)$ is singular, i.e., $\lambda_3(t)=0$, for a.e.\ $t\in[\zeta_1,\zeta_2]\subset[0,t_f]$. Recall by Lemma~\ref{rho_sing} that $\rho > 0$.  The rest of the proof can be given in two alternative ways: \\
(i) For a.e.\ $t\in[\zeta_1,\zeta_2]$: since $\lambda_3(t)=0$, one also has that $\dot\lambda_3(t)=0$; in other words, from \eqref{adjoint3a}, $\sin(\theta(t) - \phi) = 0$, which implies that $\theta(t)$ is constant, i.e., $\dot{\theta}(t) = u(t) = 0$. \\
(ii) Substituting $\rho = \lambda_0 >0$ from Lemma~\ref{rho_sing} and $\lambda_3(t)=0$ into \eqref{H_zero2}, one gets $\cos(\theta(t) - \phi) = -1$, which implies that $\theta(t)$ is constant, i.e., $\dot{\theta}(t) = u(t) = 0$, for a.e.\ $t\in[\zeta_1,\zeta_2]$.
\end{proof}

\begin{remark}  \label{u_sgn} \rm
From \eqref{control3} and Lemma~\ref{singular}, one can simply write $u(t) = -a\,\sgn(\lambda_3(t))$, a.e. $t\in[0,t_f]$. \endproof
\end{remark}

\begin{lemma}  \label{lem:lambda3_DE}
The adjoint variable $\lambda_3$ for Problem~{\em (Pc)} solves the differential equation
\begin{equation} \label{lambda3_DE} 
\dot{\lambda}_3^2(t) + \left(a\,|\lambda_3(t)| - \lambda_0\right)^2 = \rho^2\,.
\end{equation}
\end{lemma}
\begin{proof}
  From \eqref{adjoint3a},
\begin{equation} \label{eqnA} 
\dot{\lambda}_3^2(t) = \rho^2\,\sin^2(\theta(t) - \phi) = \rho^2 -
\rho^2\,\cos^2(\theta(t) - \phi)\,.
\end{equation}
Using $u(t) = -a\,\sgn(\lambda_3(t))$ in \eqref{H_zero2}, one gets
\[
\rho\,\cos(\theta(t) - \phi) = a\,|\lambda_3(t)| - \lambda_0\,.
\]
Substituting this into the right-hand side of \eqref{eqnA} and rearranging give \eqref{lambda3_DE}.
\end{proof}

In the rest of the paper, we will at times not show dependence of variables on $t$ for clarity of presentation.

\begin{figure}
\vspace*{-10mm}
\begin{center}
\psfrag{L}{$\lambda_3$}
\psfrag{Ld}{$\dot{\lambda}_3$}
\psfrag{b}{\footnotesize$-\lambda_0/a$}
\psfrag{c}{\footnotesize\hspace*{0.2mm} $\lambda_0/a$}
\psfrag{u1}{\footnotesize $u(t) = a$}
\psfrag{u2}{\footnotesize $u(t) = -a$}
\[\includegraphics[width=110mm]{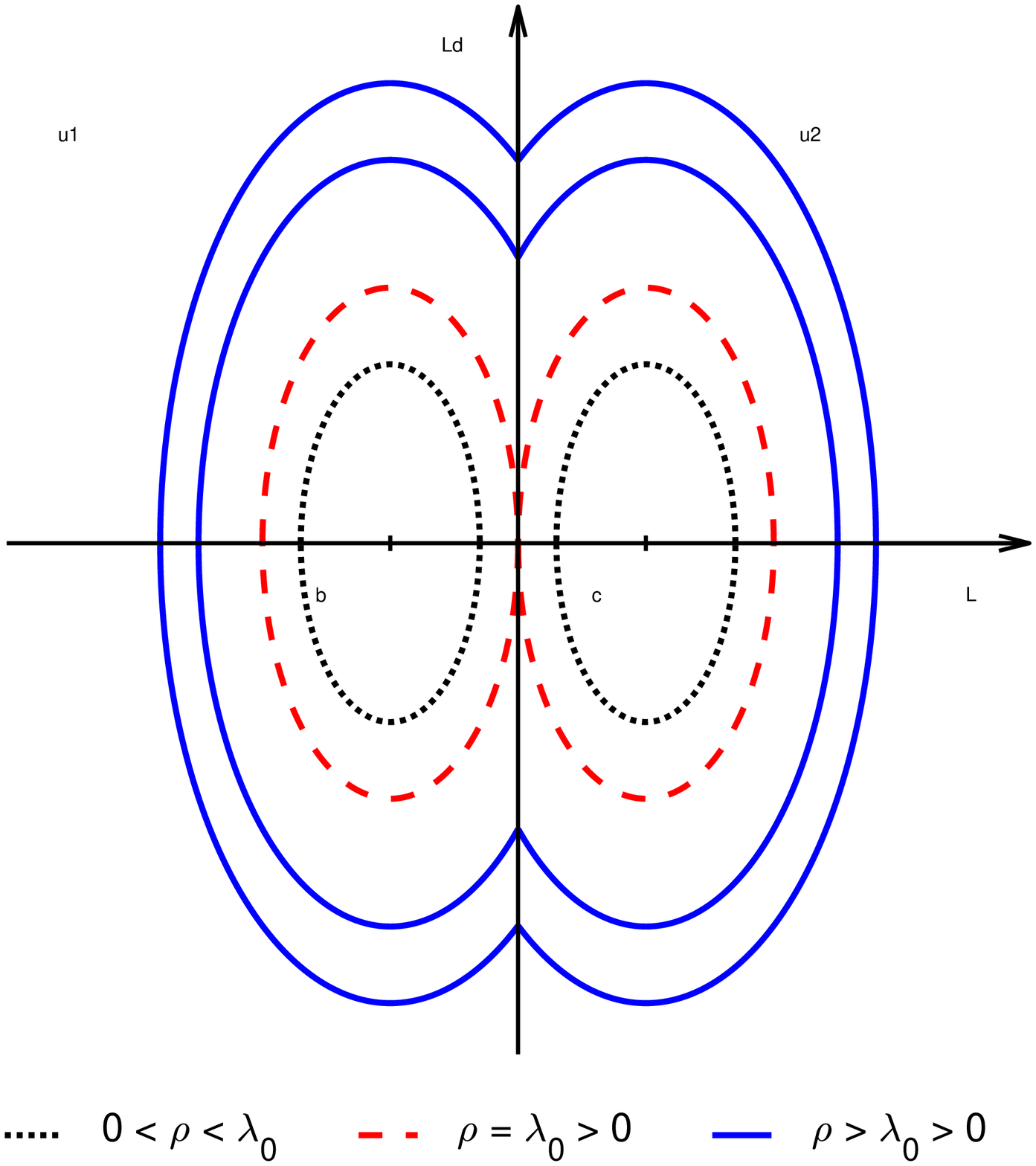}\]
\end{center}
\vspace*{-10mm}
\caption{\small\sf Phase portrait of \eqref{lambda3_DE} -- the normal case, $\lambda_0 > 0$.}
\label{phase}
\end{figure}

\begin{remark}[Normal Phase Portrait] \label{rem:ellipses} \rm 
Note that the differential equation \eqref{lambda3_DE} is given in terms of the {\em phase variables} $\lambda_3$ and $\dot{\lambda}_3$, and can be put into the form
\begin{equation} \label{lambda3_DE_ellipse} 
\left(\lambda_3 \pm \frac{\lambda_0}{a}\right)^2 +
\frac{\dot\lambda_3^2}{a^2} = \frac{\rho^2}{a^2}\,,
\end{equation}
where the ``$+$" sign in the first square term stands for the case when $\lambda_3<0$ and the ``$-$" sign for $\lambda_3>0$. Equation~\eqref{lambda3_DE_ellipse} clearly tells us that, for $\lambda_0 > 0$, the trajectories in the {\em phase plane} for $\lambda_3$ (the $\lambda_3$$\dot{\lambda}_3$-plane) will be pieces or concatenations of pieces of concentric ellipses centred at $\lambda_3 = -\lambda_0/a$ (for $u(t) = a$) and $\lambda_3 = \lambda_0/a$ (for $u(t) = -a$), as shown in Figure~\ref{phase}.  Based on the {\em phase plane diagram}, also referred to as the {\em phase portrait}, of $\lambda_3$ depicted in Figure~\ref{phase}, the following observations are made.
\begin{itemize}
\item[(i)] When $\rho>\lambda_0>0$ the trajectories are concatenations of (pieces of) ellipses, examples of which are shown by (dark blue) solid curves in Figure~\ref{phase}.  The ellipses are concatenated at the switching points $(0,\sqrt{\rho^2 - \lambda_0^2})$ and $(0,-\sqrt{\rho^2 - \lambda_0^2})$, where the value of the bang--bang control $u(t)$ switches from $a$ to $-a$ or from $-a$ to $a$, respectively.  The concatenated ellipses cross the $\lambda_3$-axis at two points, $(\lambda_0+\rho)/a$ and $-(\lambda_0+\rho)/a$.  One can promptly deduce from the diagram that if the bang--bang control has two switchings, the second arc must have a length strictly greater than $\pi/a$.  The diagram, however, does not tell as to how many switchings optimal control must have.
\item[(ii)] Recall, by Lemma~\ref{rho_sing}, that $\rho = \lambda_0 > 0$ for singular control.  The case when only a part of the trajectory is singular, referred to as a {\em bang--singular} trajectory, is represented by the two unique (red) dashed elliptic curves in Figure~\ref{phase}.  Note that singular control takes place only at the origin $(0,0)$ of the phase plane.  At any other point, the control trajectory is of bang--bang type.
\item[(iii)] For the case when $0 < \rho < \lambda_0$, example elliptic trajectories are shown with (black) dotted curves in Figure~\ref{phase}.  The trajectories cross the $\lambda_3$-axis at four distinct points $(\lambda_0\pm\rho)/a$ and $-(\lambda_0\pm\rho)/a$; however, they no longer intercept the $\dot{\lambda}_3$-axis; therefore they represent bang--bang control with no switchings, i.e., either $u(t) = a$ for all $t\in[0,t_f]$ or $u(t) = -a$ for all $t\in[0,t_f]$. \endproof
\end{itemize}
\end{remark}

\begin{lemma} \label{nonsingular}
Suppose that optimal control $u(t)$ for Problem~{\em (Pc)} is nonsingular over an interval $[\zeta_3,\zeta_4]\subset[0,t_f]$.  Then 
\begin{equation} \label{lambda3} 
|\lambda_3(t)| = \frac{1}{a}\left[\rho\,\cos(\theta(t) - \phi) + \lambda_0\right]\,,\ \ \mbox{ for a.e.\ } t\in[\zeta_3,\zeta_4]\subset[0,t_f]\,. 
\end{equation}
\end{lemma}
\begin{proof}
Substitution of\ \ $u(t) = -a\,\sgn(\lambda_3(t))$\ \ into \eqref{H_zero2} and rearranging yield the required expression. 
\end{proof}

\begin{lemma}[Nonsingular Curves] \label{rho_nonsing}
Consider Problem~{\em (Pc)} and the necessary conditions of optimality for it. \\[-6mm]
\begin{enumerate}
\item[(a)] If $\rho=0$, then $\lambda_0 > 0$ and either $u(t) = a$ or $u(t) = -a$, for all $t\in[0,t_f]$.
\item[(b)] If $\rho > 0$ and $\rho\neq\lambda_0$, then $\lambda_0 \ge 0$ and $u(t)$ is of bang--bang type.
\end{enumerate}
\end{lemma}
\begin{proof}
(a) Suppose that $\rho = 0$.  Then, from \eqref{H_zero2} and Remark~\ref{u_sgn}, $-\lambda_0 = \lambda_3(t)\,u(t) = -a\,|\lambda_3(t)| \le 0$, i.e., $\lambda_0\ge 0$.  Suppose further that $\lambda_0 = 0$.  Then we have $\lambda_3(t)\,u(t) = 0$, which means that either $\lambda_3(t) = 0$ or $u(t) = 0$, for a.e. $t\in[0,t_f]$.  If $\lambda_3(t) = 0$, then the vector of adjoint variables $\lambda(t) = {\bf 0}$, for a.e. $t\in[0,t_f]$, which is not permitted by the Pontryagin maximum principle.  Therefore $\lambda_3(t)\neq 0$.  Then, by \eqref{control3}, $u(t)\neq 0$, either.  Therefore $\lambda_0\neq 0$, namely that $\lambda_0 > 0$.  Now that $|\lambda_3(t)| = \lambda_0 / a$, $\lambda_3(t)$ is a nonzero constant, i.e., $u(t)$ is either $a$ or $-a$, for all $t\in[0,t_f]$. \\
(b) Suppose that $\rho > 0$ and $\rho\neq\lambda_0$.  We need to examine the normal and abnormal cases, separately. \\
(i) The normal case, $\lambda_0 >0$\,: We have that $0 < \rho\neq\lambda_0 > 0$.  The contrapositive of Lemma~\ref{rho_sing} states that if $0 < \rho\neq\lambda_0 > 0$ or $\rho = \lambda = 0$ then optimal control is bang--bang, which furnishes the proof for this case.  \\
(ii) The abnormal case, $\lambda_0 =0$\,: Equation \eqref{lambda3_DE} reduces to
\begin{equation}  \label{abnormal_eqn}
a^2\,\lambda_3^2 + \dot\lambda_3^2 = \rho^2\,.
\end{equation}
If $\lambda_3 =0$, then, by Equation~\eqref{abnormal_eqn}, $\dot\lambda_3 = \pm\rho \neq 0$, implying that singular optimal control is not possible.  Therefore, the optimal control is of bang--bang type.
\end{proof}

\begin{remark}  \rm
The constant optimal control in Lemma~\ref{rho_nonsing}(a), $u(t) = a$ or $u(t) = -a$, and the associated $|\lambda_3(t)| = \lambda_0 / a$, can be viewed graphically as the limiting case when $\rho\to 0$ in Figure~\ref{phase}. \endproof
\end{remark}

\begin{remark}[Abnormal Phase Portrait] \label{rem:abnormal} \rm 
Lemma~\ref{rho_nonsing}(b) implies that the optimal control formulation of the Markov-Dubins problem (Pc) can be abnormal, i.e., $\lambda_0 = 0$ is possible.  The phase portrait of the dynamical system in \eqref{abnormal_eqn}, which is depicted in Figure~\ref{abnormal_phase} for various values of $\rho$, is now comprised of concentric ellipses centred at $(0,0)$, crossing the $\dot{\lambda}_3$-axis at $\rho$ and $-\rho$, where switchings between $u(t) = a$ and $u(t) = -a$ may occur, giving rise to bang--bang control.
\begin{figure}
\vspace*{-10mm}
\begin{center}
\psfrag{L}{$\lambda_3$}
\psfrag{Ld}{\hspace*{-1mm}$\dot{\lambda}_3$}
\psfrag{u1}{\footnotesize $u(t) = a$}
\psfrag{u2}{\footnotesize $u(t) = -a$}
\[\includegraphics[width=70mm]{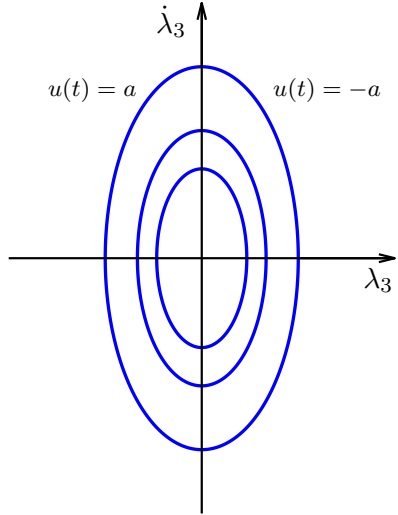}\]
\end{center}
\vspace*{-15mm}
\caption{\small\sf Phase portrait of \eqref{lambda3_DE} -- the abnormal case, $\lambda_0 = 0$.}
\label{abnormal_phase}
\end{figure}
The phase portrait also implies that the length of any bang-arc is at most $\pi/a$.  If the bang--bang control has two switchings then the length of the second arc is exactly $\pi/a$.  The portrait does not actually tell as to what the maximum number of switchings could be.  We will answer this question later in Lemma~\ref{abnormal}. \endproof
\end{remark}

We will refer to a solution to Problem~(Pc) with $\lambda_0 = 0$ as an {\em abnormal optimal solution}, and a solution with $\lambda_0 > 0$ a {\em normal optimal solution}. 

By Remarks~\ref{rem:ellipses}--\ref{rem:abnormal}, an optimal path, normal or abnormal, will in general be a concatenation of straight lines (i.e., singular arcs, where $u(t) = 0$) and circular arcs (i.e., nonsingular arcs, where $u(t) = a$ or $-a$).  In a solution trajectory, we will denote a straight line segment by an $S$ and a circular arc segment of curvature $a$ (or, turning radius $1/a$) by a $C$, resulting in descriptions of {\em optimal paths} to be {\em of type}, for example, $CSCC\cdots$, $SCS\cdots$, etc.

\begin{lemma}[Straight Line Segment] \label{straight} 
  If an optimal path for Problem~{\em (Pc)} contains a straight line segment $S$, then it is of type $CSC$, $CS$, $SC$ or $S$.
\end{lemma}
\begin{proof}
It suffices to show that an optimal path cannot contain a curve segment of any of the types $SCS$, $SCC$, and $CCS$.  Suppose that an optimal path contains the line segment $S$.  Then, the switching function $\lambda_3$ of an optimal path containing $S$ will follow the (red) dashed trajectory in the phase portrait in Figure~\ref{phase}, since it is the only trajectory which passes through the origin, where $\lambda_3 = \dot{\lambda}_3 = 0$.  Suppose that the optimal path contains a curve of type $SCS$.  Then one traces one of the two dashed ellipses passing through the origin in the phase portrait, and comes back to the origin, meaning that one whole ellipse has been traced.  In other words, the circular segment $C$ completes one full cycle, which obviously is not optimal since a full circular arc is in itself redundant.  Therefore, the optimal path cannot contain a curve of type $SCS$.  Through similar arguments (again using the non-optimality of a full circular segment), one can easily conclude that a curve of type $SCC$ or type $CCS$ in the optimal path is not optimal either, and so an optimal path cannot contain them.
\end{proof}

\subsection{Abnormal optimal solution}

By Lemmas~\ref{rho_sing} and \ref{rho_nonsing} (also see Remark~\ref{rem:abnormal}), only a bang--bang optimal solution can be abnormal.  We characterize these solutions in Lemma~\ref{abnormal} below.  In what follows, we refer to an abnormal (bang--bang) solution to Problem~(Pc) as an {\em abnormal optimal path}.

\begin{lemma}[Abnormal Markov-Dubins Curves] \label{abnormal} 
An abnormal optimal path for Problem~{\em (Pc)} is either of type $CC$ or $C$, with respective lengths of at most $2\pi/a$ and $\pi/a$.
\end{lemma}
\begin{proof}
Suppose that one has an abnormal solution to Problem~(Pc).  Then $\lambda_0 = 0$ and, by Lemma~\ref{rho_sing}, optimal control is of bang--bang type. Suppose further that optimal control has two switchings, i.e., the optimal path in the $xy$-plane is a concatenation of three circular arcs.  For simplicity, let $a=1$.  Without loss of generality, suppose that the optimal bang--bang control takes on the values, $1$, $-1$ and $1$, sequentially.  Recall from Remark~\ref{rem:abnormal} that the second bang-arc is a semi-circle and so is of length $\pi$.

Then a general configuration of the arcs will be like the one shown in Figure~\ref{abnormal_proof}, where the initial and terminal points $z_0$ and $z_f$ and the directions at these points are as indicated.  It should be noted that the given configuration is general enough, as the only other configuration which is different from the one shown is the mirror image of the diagram about the line joining the centres of the circles, corresponding to the bang-bang control sequence of $-1$, $1$ and $-1$, instead.  In Figure~\ref{abnormal_proof}, the angle $\delta$ is chosen small enough so that the (blue) dashed curve, which goes through the switching points at $(1,0)$ and $(3,0)$, is a part of the three concatenated circular arcs from $z_0$ to $z_f$.  Note that $z_0$ and $z_f$ can be placed only in the lower halves of the respective circles that they belong to, because, again by Remark~\ref{rem:abnormal}, the length of any bang-arc can be at most $\pi$.  We will show that the trajectory shown by the (blue) dashed curve is not the shortest path, in that the ``perturbed'' trajectory, which follows the (red) solid curve depicted in Figure~\ref{abnormal_proof}, is shorter.

\begin{figure}
\begin{center}
\psfrag{1}{$1$}
\psfrag{z1}{$(0,0)$}
\psfrag{z2}{$(\cos\delta, -\sin\delta)$}
\psfrag{z3}{$(4-\cos\delta, -\sin\delta)$}
\psfrag{z4}{$(4,0)$}
\psfrag{z0}{$z_0$}
\psfrag{zf}{$z_f$}
\psfrag{d}{$\delta$}
\psfrag{a}{$\beta$}
\psfrag{c}{$c$}
\psfrag{sin}{$2\,\sin\beta$}
\[\includegraphics[width=130mm]{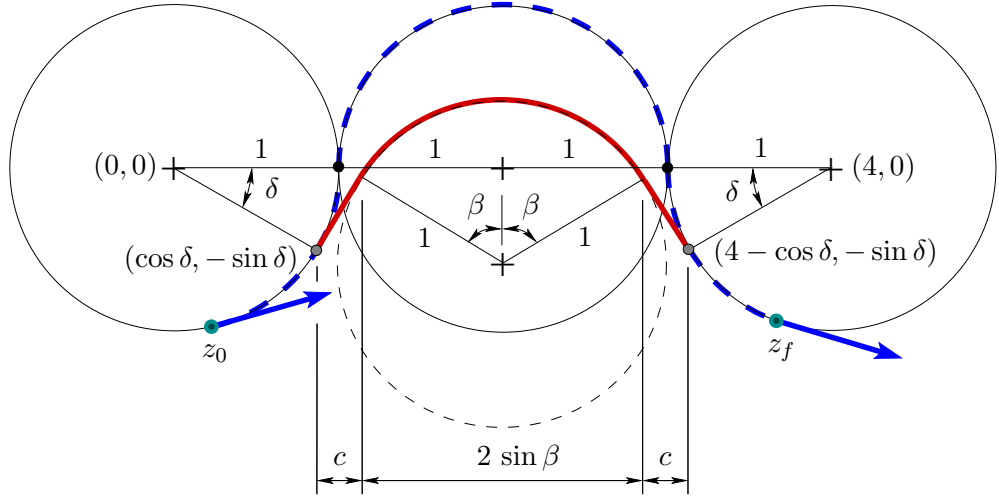}\]
\end{center}
\caption{\small\sf Diagram for the proof of Lemma~\ref{abnormal}.}
\label{abnormal_proof}
\end{figure}

The length of the dashed curve from $z_0$ to $z_f$ is simply given by
\[
\zeta = \gamma + \pi\,,
\]
where $\gamma$ is the sum of the lengths of the circular arcs from $z_0$ to $(1,0)$ and from $(3,0)$ to $z_f$.

The solid curve in the perturbed trajectory is composed by concatenating two straight line segments and a circular arc of radius $1$, in a symmetric fashion, as shown.  It is clear from the geometry that the length of each of the two straight line segments is $c\,/\sin\delta$ and the length of the circular segment is $2\,\beta$,
making the length of the solid curve $2\,(\beta + c\,/\sin\delta)$. The length of the perturbed curve can then be written as
\[
\xi = \gamma - 2\,\delta + 2\,(\beta + c\,/\sin\delta)\,,
\]
where $\beta = \pi/2 - \delta$ and $c = (4 - 2\,\cos\delta - 2\,\sin\beta) / 2$, from the given geometry.  Substituting the expressions for $c$ and $\beta$, using $\sin\beta = \cos\delta$, and rearranging further give the length of the perturbed curve as a function of $\delta$ as
\[
\xi(\delta) = \gamma + \pi - 4\,\delta + \frac{4\,(1 - \cos\delta)}{\sin\delta}\,,\quad\mbox{for } \delta\in(0,\pi/2]\,.
\]
Let $\eta(\delta) := \zeta - \xi(\delta)$.  Substitutions and trigonometric manipulations yield
\[
\eta(\delta) = 4\left(\delta - \frac{\sin\delta}{1+\cos\delta}\right)\,,\quad\mbox{for } \delta\in[0,\pi/2]\,. 
\]
Note that as $\delta\to 0^+$, $\eta(\delta) = \eta(0) = 0$, and so $\xi(0) = \zeta$.  This is in line with the geometric observation in Figure~\ref{abnormal_proof} that as $\delta\to 0^+$ the solid curve approaches the dashed curve on the upper half of the middle circle. Note also that
\[
\eta'(\delta) = 4\left(1 -
  \frac{1}{1+\cos\delta}\right)\,,\quad\mbox{for }
  \delta\in[0,\pi/2]\,,
\]
with $\lim_{\delta\to 0^+}\eta'(\delta) = \eta'(0) = 2$, and that $\eta'(\delta) = -\xi'(\delta) > 0$ for all $\delta\in[0,\pi/2)$.  These imply, along with $\xi(0) = \zeta$, that $\xi(\delta) < \zeta$, for all $\delta\in(0,\pi/2)$, furnishing the fact that an abnormal optimal path cannot be of type $CCC$.  This leaves the types $C$ and
$CC$ as the only candidates.

An abnormal optimal path of type $C$ will necessarily have a length not greater than $\pi$, as otherwise, by Remark~\ref{rem:abnormal}, one has to switch to another circular arc. By the same argument, an abnormal optimal path of type $CC$ will necessarily have a length not greater than $2\pi$.
\end{proof}

\subsection{Normal nonsingular optimal solution}

In Lemma~\ref{CCCC} below, we state that any path of type CCCC is not optimal for Problem~(Pc).  It is clear by Lemma~\ref{abnormal} that, in the abnormal case, i.e., when $\lambda_0 = 0$, Lemma~\ref{CCCC} holds immediately. For the proof of Lemma~\ref{CCCC} when $\lambda_0 > 0$ (the normal case), the diagram in Figure~\ref{CCCC_proof}, where a general configuration for a path of type CCCC and its perturbation are shown, will be used. In the figure, we have set $a=1$ for simplicity. Recall from the phase plane diagram in Figure~\ref{phase} and Remark~\ref{rem:ellipses}(i) that if the bang--bang control has two switchings, then the second arc will have a length strictly greater than $\pi$, say $\pi+\gamma$, with $\gamma > 0$. A path of type CCCC has three switchings, so the second and third arcs must both have the length $\pi+\gamma$.

\begin{figure}
\begin{center}
\psfrag{1}{$1$}
\psfrag{pc0}{$s_{c_0}$}
\psfrag{pc1}{$s_{c_1}$}
\psfrag{pc0c}{$(0,0)$}
\psfrag{pc1c}{$(2,0)$}
\psfrag{pc2}{$s_{c_2}$}
\psfrag{pc3}{$s_{c_3}$}
\psfrag{rc1}{$r_{c_1}$}
\psfrag{rc2}{$r_{c_2}$}
\psfrag{p0}{$s_0$}
\psfrag{p1}{$s_1$}
\psfrag{p2}{$s_2$}
\psfrag{r0}{$r_0$}
\psfrag{r1}{$r_1$}
\psfrag{r2}{$r_2$}
\psfrag{z0}{$z_0$}
\psfrag{zf}{$z_f$}
\psfrag{d}{$\delta$}
\psfrag{g}{$\gamma$}
\psfrag{b}{$\beta$}
\[\includegraphics[width=150mm]{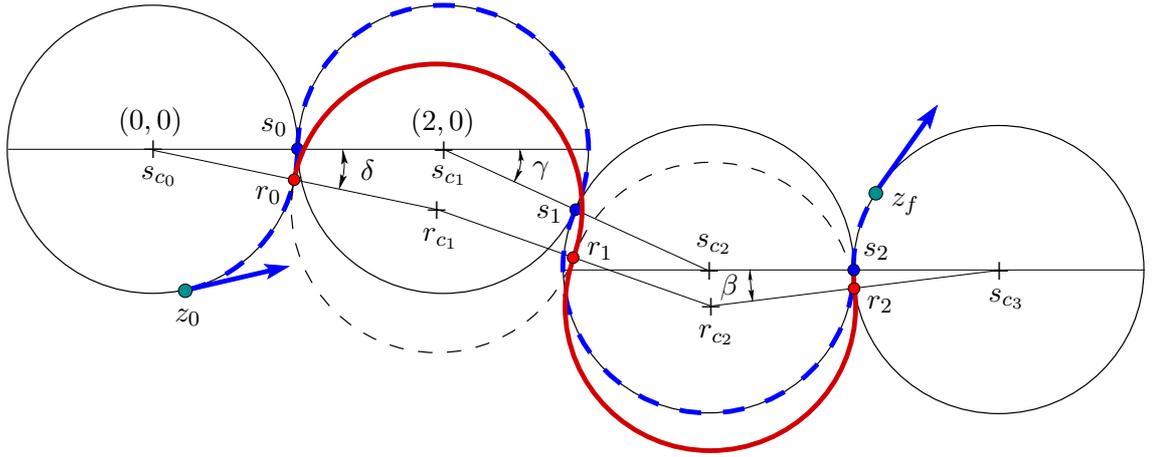}\]
\end{center}
\caption{\small\sf Diagram for the proof of Lemma~\ref{CCCC}.}
\label{CCCC_proof}
\end{figure}

Consider the path of type CCCC from $z_0$ to $z_f$, along the (blue) dashed curve, so the path satisfies the necessary optimality condition given graphically in Figure~\ref{phase}.  Let $\tau_0$ be the length of the circular arc from $z_0$ to $s_0$, $\tau_1$ the length from $s_0$ to $s_1$, $\tau_2$ the length from $s_1$ to $s_2$ and $\tau_3$ the length from $s_2$ to $z_f$.  Then the length $\tau$ of the path from $z_0$ to $z_f$ (shown by the dashed curve) with switchings from one circular subarc to another at $s_0$, $s_1$ and $s_2$ is simply given by $\tau = \tau_0 + \tau_2 + \tau_3 + \tau_3$, or
\begin{equation}  \label{tau}
\tau = \tau_0 + 2\,(\pi + \gamma) + \tau_ 3\,.
\end{equation}

It can be clearly seen in Figure~\ref{CCCC_proof} that most labelled points can be obtained from, or can be expressed in terms of, another labelled point by using rotation.  For clarity and convenience in manipulations, we will express points in the complex plane and employ {\em complex rotation}.  For example, $r_0$ can be obtained by rotating $s_0$ by angle $\delta$ in the clockwise direction about the centre $s_{c_0} = (0,0)$ of the leftmost unit circle; in other words, $r_0 = e^{-i\delta} = \cos\delta - i\,\sin\delta$, which is the {\em complex representation} of the planar point $r_0 = (\cos\delta, -\sin\delta)$.  Similarly, $s_1$ is obtained by rotating the point $(3,0)$ by angle $\gamma$ about the centre $s_{c_1}$ of a unit circle in the clockwise direction, i.e., that $s_1 = 2 + e^{-i\gamma}$ or equivalently $s_1 = (2 + \cos\gamma, -\sin\gamma)$.  By using the geometry, each labelled point in Figure~\ref{CCCC_proof} can then be expressed as follows.
\begin{equation}  \label{points1}
s_{c_0} = (0, 0)\,,\qquad
s_{c_1} = (2, 0)\,,\qquad
s_{c_2} = 2\,(1 + e^{-i\gamma})\,,\qquad 
s_{c_3} = 2\,(2 + e^{-i\gamma})\,,
\end{equation}
\begin{equation}  \label{points2}
s_1 = 2 + e^{-i\gamma}\,,\qquad
s_2 = 3 + 2\,e^{-i\gamma}\,,
\end{equation}
\begin{equation}  \label{points3}
r_0 = e^{-i\delta}\,,\qquad
r_{c_1} = 2\,e^{-i\delta}\,,\qquad
r_{c_2} = 2\,(2 + e^{-i\gamma} - e^{i\beta})\,,\qquad
r_2 = 4 + 2\,e^{-i\gamma} - e^{i\beta}\,.
\end{equation}
Note that $r_1 = (r_{c_1} + r_{c_2}) / 2$. Then
\begin{equation}  \label{points4}
r_1 = 2 + e^{-i\gamma} + e^{-i\delta} - e^{i\beta}\,.
\end{equation}

\begin{lemma}[Non-optimality of $CCCC$ Type] \label{CCCC} 
Any path of type CCCC is not optimal for Problem~{\em (Pc)}.
\end{lemma}
\begin{proof}
As discussed before, a general configuration for a path of type CCCC is depicted in Figure~\ref{CCCC_proof}.  Recall that the coordinates of the labelled points in Figure~\ref{CCCC_proof} have been given (in the complex plane) in \eqref{points1}--\eqref{points4}.

We will show that the path of type CCCC between the oriented points $z_0$ and $z_f$ in Figure~\ref{CCCC_proof} is not optimal.  Note that the point $z_0$ can be taken to be anywhere on the lower semi-circle where it currently sits.  The perturbation $\delta>0$ is arbitrarily small so that $r_0$ is {\em between} $z_0$ and $s_0$.  Similarly, the point $z_f$ can be taken to be anywhere on the upper semi-circle where it currently resides.

The length $\eta$ of the perturbed path can now be written as a function of $\delta$ as follows.
\begin{equation}  \label{eta}
\eta(\delta) = \tau_0 - \delta + \xi_1(\delta)  + \xi_2(\delta) + \beta(\delta)  + \tau_3\,, 
\end{equation}
where $\xi_1$ is the length of the circular path from $r_0$ to $r_1$, $\xi_2$ the length of the circular path from $r_1$ to $r_2$, and $\beta$ the length of the circular path from $r_2$ to $s_2$, each of which are functions of $\delta$.  Note that
\begin{equation}  \label{initial}
\xi_1(0) = \xi_2(0) = \pi + \gamma\,,\quad \beta(0) = 0\,,
\end{equation}
and so, using \eqref{tau},
\[
\eta(0) = \tau\,.
\]
Then it suffices to show that $\eta'(0) < 0$.  If, however, $\eta'(0) = 0$, as it will turn out to be the case, it will
consequently suffice to show that $\eta''(0) < 0$.  From \eqref{eta}, one has that
\begin{equation}  \label{etad0}
\eta'(0) = -1 + \xi_1'(0)  + \xi_2'(0)  + \beta'(0) \,.
\end{equation}
In what follows, we will not always show dependence of $\xi_1$, $\xi_2$ and $\beta$ on $\delta$, for clarity.  It is straightforward to write, from Figure~\ref{CCCC_proof}, that
\begin{eqnarray}
|r_{c_1} - r_{c_2}|^2 &=& 4\,, \label{eqn1a} \\[1mm]
 e^{-i\xi_1} (r_0 - r_{c_1}) &=& r_1 - r_{c_1}\,, \label{eqn2a} \\[1mm]
e^{i\xi_2} (r_1 - r_{c_2}) &=& r_2 - r_{c_2}\,. \label{eqn3a}
\end{eqnarray}
After substitutions and simplifying manipulations, \eqref{eqn1a}--\eqref{eqn3a} yield the following equations,
respectively.
\begin{eqnarray}
&& 2\,(\cos\delta + \cos\beta - \cos\gamma) - \cos(\delta+\beta) + \cos(\delta-\gamma) - \cos(\beta+\gamma) = 3\,, \label{eqn1b} \\[1mm] 
&& e^{-i\xi_1} = 1 - 2\,e^{i\delta} - e^{i(\delta - \gamma)} + e^{i(\delta+\beta)}\,, \label{eqn2b} \\[1mm]
&& e^{-i\xi_2} = 1 - 2\,e^{-i\beta} - e^{-i(\beta + \gamma)} + e^{-i(\delta+\beta)}\,. \label{eqn3b}
\end{eqnarray}
When $\delta = 0$ (the unperturbed case), Equation~\eqref{eqn1b} is verified by $\cos(\beta) = 1$, i.e., $\beta = 0$, Equation~\eqref{eqn2b}~reduces to $e^{-i\xi_1(0)} = -e^{-i\gamma} = e^{-i(\pi+\gamma)}$, i.e., $\xi_1(0) = \pi+\gamma$, and Equation~\eqref{eqn3b} to $e^{-i\xi_2(0)} = -e^{-i\gamma} = e^{-i(\pi+\gamma)}$, i.e., $\xi_2(0) = \pi+\gamma$, which altogether reconfirm \eqref{initial}.

In order to evaluate $\eta'(0)$ in \eqref{etad0}, we need $\xi_1'(0)$, $\xi_2'(0)$ and $\beta'(0)$, which can be obtained by differentiating Equations~\eqref{eqn1b}--\eqref{eqn3b} and substituting the known quantities for the unperturbed case.  Differentiating \eqref{eqn1b} with respect to $\delta$, one gets
\begin{equation}  \label{eqn1deriv}
2\,(\sin\delta + \beta'\,\sin\beta) - (1 +
\beta')\,\sin(\delta+\beta) + \sin(\delta-\gamma) +
\beta'\,\sin(\beta+\gamma) = 0\,.
\end{equation}
Then substitution of $\delta = \beta = 0$ (the unperturbed case) into \eqref{eqn1deriv} and the fact that $\sin\gamma\neq0$ yield
\begin{equation}  \label{betad0}
\beta'(0) = 1\,.
\end{equation}
Differentiation of Equation~\eqref{eqn2b} with respect to $\delta$ gives
\begin{equation}  \label{eqn2deriv}
\xi_1'\,e^{-i\xi_1} = 2\,e^{i\delta} + e^{i(\delta - \gamma)} - (1 +
\beta')\,e^{i(\delta+\beta)}\,.
\end{equation}
Substitution of $\delta = \beta = 0$, $\beta'(0) = 1$ and $\xi_1(0) = \pi+\gamma$, for the unperturbed case, into Equation~\eqref{eqn2deriv} and manipulations result in
\begin{equation}  \label{xi1d0}
\xi_1'(0) = -1\,.
\end{equation}
Similarly, differentiation of Equation~\eqref{eqn3b} with respect to $\delta$ results in
\begin{equation}  \label{eqn3deriv}
\xi_2'\,e^{-i\xi_2} = -2\,\beta'\,e^{-i\beta} -
\beta'\,e^{-i(\beta+\gamma)} + (1 + \beta')\,e^{i(\delta+\beta)}\,,
\end{equation}
and the substitution of $\delta = \beta = 0$, $\beta'(0) = 1$ and $\xi_2(0) = \pi+\gamma$, and manipulations give
\begin{equation}  \label{xi2d0}
\xi_2'(0) = 1\,.
\end{equation}
Now, substituting \eqref{betad0}, \eqref{xi1d0} and \eqref{xi2d0} into \eqref{etad0}, one gets
\[
\eta'(0) = 0\,.
\]
As pointed out immediately after~\eqref{initial} above, this necessitates to check if
\begin{equation}  \label{etadd}
\eta''(0) = \xi_1''(0)  + \xi_2''(0)  + \beta''(0) < 0\,
\end{equation}
to furnish the proof. Now, differentiate the equation in
\eqref{eqn1deriv} to get
\begin{eqnarray}
&& 2\,(\cos\delta + \beta''\,\sin\beta + (\beta')^2\cos\beta) - \beta''\,\sin(\delta+\beta) - (1+\beta')^2\cos(\delta+\beta) \nonumber \\[1mm]
&& \hspace*{45mm} + \cos(\delta-\gamma) + \beta''\,\sin(\beta+\gamma) + (\beta')^2\cos(\beta+\gamma) = 0\,.  \label{eqn1deriv2}
\end{eqnarray}
Substituting into \eqref{eqn1deriv2} the (unperturbed) quantities $\delta = \beta = 0$ and $\beta'(0) = 1$, and manipulating further, one gets
\begin{equation}  \label{betadd}
\beta''(0) = -2\,\cot\gamma\,,
\end{equation}
Recall that $\gamma>0$, so \eqref{betadd} is well-defined.  Next, differentiate \eqref{eqn2deriv} to get
\begin{equation}  \label{eqn2deriv2}
\left(\xi_1'' - i\,(\xi_1')^2\right)\,e^{-i\xi_1} = 2i\,e^{i\delta} +
i\,e^{i(\delta - \gamma)} - \left(\beta'' + i\,(1 +\beta')^2\right)\,e^{i(\delta+\beta)}\,.
\end{equation}
Substituting into \eqref{eqn2deriv2} $\delta = \beta = 0$, $\beta'(0) = 1$, $\beta''(0) = -2\,\cot\gamma$, $\xi_1(0) = \pi+\gamma$ and $\xi_1'(0) = -1$, and manipulating further, one gets
\[
\xi_1''\, e^{-i\gamma} = -2\,\cot\gamma + 2i\,,
\]
which yields
\begin{equation}  \label{xi1dd}
\xi_1''(0) = -2\,\csc\gamma\,.
\end{equation}
Similarly, differentiation of \eqref{eqn3deriv} gives
\begin{eqnarray}
\left(\xi_2'' - i\,(\xi_2')^2\right)\,e^{-i\xi_2} &=& -\left(\beta'' - i\,(\beta')^2\right)\left(2\,e^{-i\beta} +
    e^{-i(\beta+\gamma)}\right)  \nonumber \\[1mm]
&& +\ \left(\beta'' + i\,(1+\beta')^2\right)\,e^{-i(\delta+\beta)}\,.  \label{eqn3deriv2}
\end{eqnarray}
Substitution of $\beta'(0) = 1$, $\beta''(0) = -2\,\cot\gamma$, $\xi_2(0) = \pi+\gamma$ and $\xi_2'(0) = 1$ into \eqref{eqn3deriv2}, and some lengthy manipulations, result in
\begin{equation}  \label{xi2dd}
\xi_2''(0) = -2\,\cot\gamma\,.
\end{equation}
Substitution of \eqref{betadd}, \eqref{xi1dd} and \eqref{xi2dd} into \eqref{etadd} yields
\[
\eta''(0) = -2\left(\frac{1+2\,\cos\gamma}{\sin\gamma}\right) < 0\,,
\]
for any $\gamma\in(0,\pi/2]$.  Therefore, given the fact that $\eta(0) =\tau$ and $ \eta'(0) = 0$, we conclude that $\eta$ is decreasing at $\delta=0$ and, by continuity of $\eta$, the perturbed length $\eta(\delta) < \tau$ for small enough $\delta > 0$.  This completes the proof.
\end{proof}

\subsection{Dubins' result}

The preceding results (Lemmas~\ref{rho_nonsing}--\ref{CCCC} and the phase plane diagram in Figure~\ref{phase}) obtained by using Pontryagin maximum principle and perturbation of trajectories can now be used to prove  Dubins' result in~\cite{Dubins1957}.

\begin{theorem}[Markov-Dubins Curves -- Dubins~\cite{Dubins1957}]  \label{Dubins}
Any solution of Problem~{\em (P)}, that is, any $C^1$ and piecewise-$C^2$ shortest path of bounded curvature in the plane between two prescribed endpoints, where the slopes of the path are also prescribed, is of type $CSC$, or of type $CCC$, or a subset thereof.  Moreover, if the shortest path is of type $CCC$, then the second circular arc is of length greater than $\pi/a$.
\end{theorem}
\begin{proof}
If the solution is abnormal, i.e., $\lambda_0 = 0$, then by Lemma~\ref{abnormal} the shortest path is of type $C$ or $CC$, which in either case is a subset of $CCC$.  Suppose that the solution is normal, i.e., $\lambda_0 \neq 0$.  Then the shortest path is either of type
\begin{enumerate}
\item[(i)] $CSC$, $CS$, $SC$ or $S$, if it contains a straight line segment, by Lemma~\ref{straight}, or
\item[(ii)] $CCC$, $CC$ or $C$, by Lemmas~\ref{rho_nonsing} and~\ref{CCCC}.
\end{enumerate}
The last statement of the theorem is proved by using the phase portrait in Figure~\ref{phase}, with $\rho>\lambda_0>0$.  If the shortest path is of type $CCC$, then three pieces of ellipses in Figure~\ref{phase} are concatenated, with the second ellipse sweeping an angle greater than $\pi$, completing the proof.
\end{proof}

\section{Stationarity of feasible curves of type $CSC$ or $CCC$}

Dubins' result in Theorem~\ref{Dubins} states that if a curve is optimal for Problem~(P), or equivalently for Problem~(Pc), then that curve is of type $CSC$ or $CCC$, or a subset thereof.  It is clear that a solution curve of Problem~(Pc) verifies the Pontryagin maximum principle, which furnish necessary conditions of optimality.  However, neither \cite{Dubins1957} nor \cite{BoiCerLeb1991,SusTan1991} provide any further explanation as to whether or not any other feasible curve for Problem~(Pc), i.e., any other curve that satisfies the constraints of Problem~(Pc), which are of the types listed above,  would also verify the Pontryagin maximum principle.  These feasible curves presumably have a length larger than that of a solution curve of Problem~(Pc).  In Theorem~\ref{stationarity} below, we state that these kinds of feasible solutions indeed verify the necessary conditions of optimality, i.e., they are {\em stationary}, or {\em critical}, solutions of Problem~(P).

In the proof of Theorem~\ref{stationarity}, we assume that a feasible solution has been provided, in that the times $t_1$ and $t_2$ at which switchings from one subarc to another occur (note that in general $0\le t_1\le t_2\le t_f$), as well as the terminal time $t_f$, which is the length of the curve, are known.  The signed curvatures ($u(t) = a$ or $-a$) of the $C$ subarcs of the feasible curve are also given, and so $\theta_1 := \theta(t_1)$ and $\theta_2 := \theta(t_2)$ are also known/easily calculable.  Recall from the original problem description that $\theta(0) = \theta_0$.  \\[-8mm]
\begin{theorem}[Stationarity of Feasible Curves]  \label{stationarity}
Any feasible path for Problem~{\em (Pc)}, i.e., any path satisfying the constraints of Problem~{\em (Pc)}, which is of type $CSC$ or $CCC$, or a subset thereof, verifies the Pontryagin maximum principle. \\[-5mm]
\end{theorem}
\begin{proof}
It suffices to show that, for any feasible curve of type $CSC$ or $CCC$, or a subset thereof, there exist adjoint variables $\lambda_0 \ge 0$, $\lambda_1$, $\lambda_2$ and $\lambda_3$ such that \eqref{H_zero2} holds.  Recall that the constants $\rho$ and $\phi$ in \eqref{H_zero2} are defined in \eqref{rhophi} in terms of the constant values of $\lambda_1$ and $\lambda_2$.  So the task boils down to finding pertaining $\rho$, $\phi$ and $\lambda_3$ for every single possible feasible curve of the given types.

First, suppose that Problem~(Pc) is normal, i.e., $\lambda_0 > 0$, and without loss of generality, set $\lambda_0 = 1$.  We examine feasible curves of the two basic types, $CSC$ or $CCC$, and their subsets thereof, for the normal case, one by one. \\[-8mm]
\begin{itemize}
\item[(a)]  Consider a feasible curve of type $CSC$, or of one of the types $CS$, $SC$, and $S$.  Recall that along the subarc $S$, $\lambda_3(t) = 0$, and so $\rho = \lambda_0 = 1$ by Lemma~\ref{rho_sing}.  Hence, Equation~\eqref{H_zero2} reduces to $\cos(\theta(t) - \phi) = -1$, with $\theta(t) = \theta_1$ constant, which implies that $\phi = \theta_1 \pm \pi$.  Set $\phi = \theta_1 - \pi$.  Then, for a feasible curve of type $CSC$, the adjoint variable, or the switching function, $\lambda_3(t)$ can be constructed {\em uniquely}, in terms of the feasible solution parameters $t_1$, $t_2$, $t_f$ and $\theta_1$, as
\begin{equation}  \label{lambda3_CSC}
\lambda_3(t) = \left\{\begin{array}{ll}
-\left[\cos(\theta(t) - \theta_1 + \pi) + 1\right]/u(t)\,, & \mbox{\ if\ \ } 0 < t < t_1\mbox{\ \ or\ \ } t_2 < t < t_f\,, \\[2mm]
0\,, & \mbox{\ if\ \ } t_1 < t < t_2\,,
\end{array}\right.
\end{equation}
We note that the first equation in \eqref{lambda3_CSC} can be re-written, by using $u(t) = -a\,\sgn(\lambda_3(t))$, as
\[
|\lambda_3(t)| = \left[\cos(\theta(t) - \theta_1 + \pi) + 1\right]/a\,,
\]
which agrees with \eqref{lambda3}.  The expression in \eqref{lambda3_CSC} holds for any feasible curve of type $CS$ (where $t_2 = t_f$), type $SC$ (where $t_1 = 0$) or type $S$ (where $t_1 = 0$ and $t_2 = t_f$), too.

\item[(b)]  Consider feasible curves of each of the types $CCC$, $CC$ and $C$, one by one, below.
\begin{itemize}
\item[(i)]  Type $CCC$:  This type requires two switchings, so $0 < t_1 < t_2 < t_f$ and that $\lambda_3(t_1) = 0$ and $\lambda_3(t_2) = 0$, with which Equation~\eqref{H_zero2} yields two equations in the two unknowns $\rho$ and $\phi$; namely, $\rho\,\cos(\theta_1 - \phi) + 1 = 0$ and $\rho\,\cos(\theta_2 - \phi) + 1 = 0$, where $\rho > 1$ by Remark~\ref{rem:ellipses}.  These two equations result in $\cos(\theta_1 - \phi) = \cos(\theta_2 - \phi)$.  By Figure~\ref{phase} and the second statement of Theorem~\ref{Dubins}, $\theta_2 - \phi = -(\theta_1 - \phi)$.  Then simple algebraic manipulations provide a unique solution for the constants $\rho$ and $\phi$ as:
\[
\phi = (\theta_1 + \theta_2)/2\,\quad\mbox{and}\quad
\rho = -\sec\left((\theta_1 - \theta_2)/2\right)\,.
\]
Here, one can easily verify that $\rho>1$, indeed, as follows: since $|\theta_1 - \theta_2| > \pi$ by the second statement of Theorem~\ref{Dubins}, $-1 < \cos\left((\theta_1 - \theta_2)/2\right) < 0$ and so $-\sec\left((\theta_1 - \theta_2)/2\right) > 1$.  Finally, a {\em unique} solution for $\lambda_3(t)$, in terms of the feasible solution parameters $t_1$, $t_2$, $t_f$, $\theta_1$ and $\theta_2$, can be written down for this case as
\begin{equation}  \label{lambda3_CCC}
\lambda_3(t) = -\sec\left((\theta_1 - \theta_2)/2\right)\left[\cos(\theta(t) - (\theta_1 + \theta_2)/2) + 1\right]/u(t)\,,\ \mbox{a.e. } t\in[0, t_f]\,.
\end{equation}
With $u(t) = -a\,\sgn(\lambda_3(t))$, Equation~\eqref{lambda3_CCC} becomes
\[
|\lambda_3(t)| = \sec\left((\theta_1 - \theta_2)/2\right)\left[\cos(\theta(t) - (\theta_1 + \theta_2)/2) + 1\right]/a\,,
\]
which agrees with \eqref{lambda3}.
\item[(ii)]  Type $CC$: This type requires only one switching, so, without loss of generality, let $0 < t_1 < t_2 = t_f$.  Then $\lambda_3(t_1) = 0$ and Equation~\eqref{H_zero2} result in $\rho\,\cos(\theta_1 - \phi) + 1 = 0$, where $\rho > 1$ by Remark~\ref{rem:ellipses}.  This single equation in two unknowns results in infinitely many solutions for $\rho > 1$ and $\phi$.  For simplicity, take $\rho =2$ and so $\phi = \theta_1 \pm 2\pi/3$, or simply take $\phi = \theta_1 - 2\pi/3$.  Then $\lambda_3(t)$ is given as
\begin{equation}  \label{lambda3_CC}
\lambda_3(t) = -2\left[\cos(\theta(t) - \theta_1 + 2\pi/3) + 1\right]/u(t)\,,\quad\mbox{a.e. } t\in[0, t_f]\,.
\end{equation}
For verification purposes, it can be easily seen that
\[
|\lambda_3(t)| = 2\left[\cos(\theta(t) - \theta_1 + 2\pi/3) + 1\right]/a\,,
\]
which agrees with \eqref{lambda3}.
\item[(iii)]  Type $C$:  This type requires no switchings, so by Remark~\ref{rem:ellipses}, any $0 < \rho < 1$ and any $-\pi\le\phi\le\pi$ would do.  Take $\rho = 1/2$ and $\phi = 0$, for the sake of simplicity, which yield
\begin{equation}  \label{lambda3_C}
\lambda_3(t) = -\left[\cos\theta(t) + 1\right]/(2\,u(t))\,,\quad\mbox{a.e. } t\in[0, t_f]\,.
\end{equation}
From \eqref{lambda3_C}, one gets $|\lambda_3(t)| = \left[\cos\theta(t) + 1\right]/(2a)$, which agrees with \eqref{lambda3}. \\[-8mm]
\end{itemize}
\end{itemize}
Next, suppose that Problem~(Pc) is abnormal, i.e., $\lambda_0 = 0$.  Then, by Lemma~\ref{abnormal}, the only possible types are $C$ and $CC$, which we consider below. \\[-8mm]
\begin{itemize}
\item[(a)]  Type $CC$:  In this case, by \eqref{abnormal_eqn} and Figure~\ref{abnormal_phase}, $|\theta_0 - \theta_1| \le \pi$ and $|\theta_1 - \theta_f| \le \pi$.  This type has only one switching, so without loss of generality, let $0\neq t_1\neq t_2 = t_f$.  Then $\lambda_3(t_1) = 0$ and Equation~\eqref{H_zero2} give $\rho\,\cos(\theta_1 - \phi) = 0$, where $\rho > 0$ by Remark~\ref{rem:abnormal}, and so $\cos(\theta_1 - \phi) = 0$, which yields $\phi = \theta_1 \pm \pi/2$.  For simplicity, take $\rho = 1$ and $\phi = \theta_1 - \pi/2$. Now, one can write
\begin{equation}  \label{lambda3_CC_abnormal}
\lambda_3(t) = -\cos(\theta(t) - \theta_1 + \pi/2)/u(t)\,,\quad\mbox{a.e. } t\in[0, t_f]\,.
\end{equation}
From \eqref{lambda3_CC_abnormal}, one gets $|\lambda_3(t)| = \cos(\theta(t) - \theta_1 + \pi/2)/a$, which agrees with \eqref{lambda3}.
\item[(b)]  Type $C$:  In this case, by \eqref{abnormal_eqn} and the phase plane diagram in Figure~\ref{abnormal_phase}, $|\theta_0 - \theta_1| \le \pi$.  This type has no switchings, so $\rho > 0$ can be chosen arbitrarily and $\phi$ can be determined based on the value of $\rho$.  For simplicity take $\rho = 1$.  Then
\[
\lambda_3(t) = -\cos(\theta(t) - \phi)/u(t)\,,\quad\mbox{a.e. } t\in[0, t_f]\,,\ \mbox{ or}
\]
\[
|\lambda_3(t)| = \cos(\theta(t) - \phi)/a\,,\quad\mbox{a.e. } t\in[0, t_f]\,,
\]
which agrees with \eqref{lambda3} if $\theta_0 - \phi = - \pi/2$, whenever $u(t) \equiv a$, or if $\theta_0 - \phi = \pi/2$, whenever $u(t) \equiv -a$.  Therefore a valid construction of $\lambda_3(t)$ is as follows. For a.e. $t\in[0, t_f]$,
\begin{equation}  \label{lambda3_C_abnormal}
\lambda_3(t) = \left\{\begin{array}{ll}
-\cos(\theta(t) - \theta_0 - \pi/2)/a\,, & \mbox{ if } u(t) = a\,, \\[1mm]
\cos(\theta(t) - \theta_0 + \pi/2)/a\,, & \mbox{ if } u(t) = -a\,.
\end{array} \right.
\end{equation}
\end{itemize}
This completes the proof.
\end{proof}

\begin{remark}  \label{rem:normal_abnormal} \rm
The proof of Theorem~\ref{stationarity} provides closed form expressions for the adjoint variable $\lambda_3$, which is also the switching function, once a  feasible curve of certain types, of the Markov-Dubins problem, is given.  Expression~\eqref{lambda3_CSC} is provided for feasible curves of types $CSC$, $CS$, $SC$ and $S$.  In the case when Problem~(Pc) is normal, Expression \eqref{lambda3_CCC} is derived for type $CCC$, \eqref{lambda3_CC} for type $CC$ and \eqref{lambda3_C} for type $C$.  In the case when Problem~(Pc) is abnormal, Expression \eqref{lambda3_CC_abnormal} is provided for type $CC$ and \eqref{lambda3_C_abnormal} for type $C$.  It is interesting to note that, in the case when $|\theta_0 - \theta_1| \le \pi$ and $|\theta_1 - \theta_f| \le \pi$, either of the expressions \eqref{lambda3_CC} (with $\lambda_0 > 0$) and \eqref{lambda3_CC_abnormal} (with $\lambda_0 = 0$) can be used, meaning that normal and abnormal adjoint variable solutions exist concurrently.  The same comment holds for the expressions \eqref{lambda3_C} (with $\lambda_0 > 0$) and \eqref{lambda3_C_abnormal}(with $\lambda_0 = 0$), in the case when $|\theta_0 - \theta_1| \le \pi$.  It should, however, be noted that, if $|\theta_0 - \theta_1| > \pi$ or $|\theta_1 - \theta_f| > \pi$, then the solution cannot be abnormal.  These comments justify the following two corollaries.
\endproof
\end{remark}

\begin{corollary}[Equivalence of Adjoint Variables in Normal and Abnormal Cases]  \label{normal_abnormal_1}
\ \vspace*{-3mm}
\begin{itemize}
\item[(a)] Suppose that a feasible solution of type $C$ of Problem~{\em (Pc)} is given.  If $|\theta_0 - \theta_1| \le \pi$, then the adjoint variable $\lambda_3$ exists with any $\lambda_0\ge 0$.  In particular, one has \eqref{lambda3_C} with $\lambda_0 = 1$, and \eqref{lambda3_C_abnormal} with $\lambda_0 = 0$.
\item[(b)] Suppose that a feasible solution of type $CC$ of Problem~{\em (Pc)} is given.  If $|\theta_0 - \theta_1| \le \pi$ and $|\theta_1 - \theta_f| \le \pi$, then the adjoint variable $\lambda_3$ exists with any $\lambda_0\ge 0$.  In particular, one has \eqref{lambda3_CC} with $\lambda_0 = 1$, and \eqref{lambda3_CC_abnormal} with $\lambda_0 = 0$.
\end{itemize}
\end{corollary}
Thus we can state: \\[-7mm]
\begin{corollary}[Normality of Abnormal Curves]  \label{normal_abnormal_2}
If Problem~{\em (Pc)} has an abnormal solution, then that abnormal solution (or path) is also a normal solution.  The converse, however, does not hold.
\end{corollary}

\section{A Numerical Method for Markov-Dubins Curves}
\label{MD_num}

Theorem~\ref{Dubins} facilitates the employment of an efficient numerical scheme for finding a Markov-Dubins curve, or path, as it characterizes all of the types of solutions, where each subarc solution can be written analytically, and so the infinite-dimensional optimization problem~(P) is reduced to a combinatorial decision making with a relatively small number of choices.  Theorem~\ref{Dubins} asserts that a Markov-Dubins path will be of one of the types $CSC$ and $CCC$, or a subset of these types, giving rise to seven possibilities, as listed in (i) and (ii) in the proof of Theorem~\ref{Dubins}. Recall that a $C$ subarc may happen to be a {\em left-turn} circular arc, denoted~$L$, where $\dot\theta(t)>0$, or a {\em right-turn} circular arc, denoted~$R$, where $\dot\theta(t)<0$.  This increases the number of possibilities: By using the signed curvature, Markov-Dubins path can be characterized as one of particular types listed in the set
\begin{equation}  \label{LSR}
\{LRL, RLR, LSL, LSR, RSL, RSR\}
\end{equation}
or a subset of one of the types in the list, altogether yielding 15 possibilities.  All of these possible types, including, for example, the types $LR$, $RL$, $RS$, $L$, and so on, can be checked algebraically to find all feasible paths and select the shortest of them---see, for example, \cite{ShkLum2001}, for a classification and identification of these path types.  We propose an alternative numerical approach to identifying the type of the path and computing the duration of each component, or subarc, of the path.  Ultimately, the approach we present here should serve as a building block for a numerical approach for constructing Markov-Dubins interpolating curves, in which case, the much larger number of possible types of curves is prohibitive for an algebraic approach.

The technique we propose involves parameterization of the shortest path problem with respect to the terminal and switching times, i.e., the times at which the path switches from one subarc to another, say, from $L$ to $S$.  This approach is markedly different from previous approaches to computing the Markov-Dubins path in the literature.  Efficient computation of switching times, referred to as switching time optimization, for optimal control problems have been extensively studied in the past---see~\cite{KayNoa2003,MauBueKimKay2005,KayMau2014,Vossen2006} and the references therein.  In comparison with these general studies, we do not have to integrate numerically the ODEs involved.  We can write down a concatenation of the analytic solutions for the state and control variables along each subarc in terms of the switching times---this avoids the task of discretizing the problem and thus makes computations fast---see~\cite{Kaya2010,BanKay2013} for examples of discretization of general optimal control problems.

Let $L_{\xi_1}$ denote a left-turn arc with length $\xi_1$, $R_{\xi_2}$ a right-turn arc with length $\xi_2$ and $S_{\xi_3}$ a straight line with length $\xi_3$.  Similarly, let $L_{\xi_4}$ denote a left-turn arc with length $a\xi_4$ and $R_{\xi_5}$ a right-turn arc with length $\xi_5$.  Then the types of solution arcs described  in Theorem~\ref{Dubins} or the more particular types of solution arcs given in \eqref{LSR} can all be represented by the string
\[
L_{\xi_1}R_{\xi_2}S_{\xi_3}L_{\xi_4}R_{\xi_5}\,.
\]
For example, the type $RLR$ is given with $\xi_1=\xi_3=0$ and $\xi_2,\xi_4,\xi_5 > 0$. On the other hand, the type $LR$ is obtained either with $\xi_3=\xi_4=\xi_5=0$ and $\xi_1,\xi_2 > 0$ or with $\xi_1=\xi_2=\xi_3=0$ and $\xi_4,\xi_5 > 0$.

Let the initial time $t_0 := 0$ and the terminal time $t_5 := t_f$.  Also define the {\em switching times} $t_j$, $j = 1,\ldots 4$, such that
\begin{equation}  \label{arc_durations}
\xi_j := t_j - t_{j-1}\,,\quad\mbox{for } j = 1,\ldots,5\,.
\end{equation}
The lengths $\xi_j$ may also be referred to as {\em arc durations}.  Note that along an arc of type $L$ or $R$ the optimal control is bang--bang; in particular, $u(t) \equiv a$ along an arc of type $L$ and $u(t) \equiv -a$ along an arc of type $R$.  Along an arc of type $S$, on the other hand, optimal control is singular; in particular, $u(t) \equiv 0$.  Then the ODEs in Problem~(Pc) can be solved as follows.  For $t_{j-1} \le t < t_j$,
\begin{eqnarray}
&& \theta(t) = \theta(t_{j-1}) + u(t)\,(t - t_{j-1})\,,\quad\mbox{ if } j = 1,\ldots,5\,, \label{theta_arc} \\[2mm]
&& x(t) = \left\{\begin{array}{ll}
x(t_{j-1}) + (\sin\theta(t) - \sin\theta(t_{j-1})) / u(t)\,, & \mbox{ if } j = 1,2,4,5\,, \\[1mm]
x(t_{j-1}) + \cos\theta(t)\,(t - t_{j-1})\,, & \mbox{ if } j = 3\,,
\end{array}\right. \label{x_arc} \\[2mm]
&& y(t) = \left\{\begin{array}{ll}
y(t_{j-1}) + (\cos\theta(t) - \cos\theta(t_{j-1})) / u(t)\,, & \mbox{ if } j = 1,2,4,5\,, \\[1mm]
y(t_{j-1}) + \sin\theta(t)\,(t - t_{j-1})\,, & \mbox{ if } j = 3\,, \label{y_arc}
\end{array}\right.
\end{eqnarray}
where
\begin{equation}  \label{arc_control}
u(t) = \left\{\begin{array}{rl}
a\,, & \mbox{ if } j = 1,4\,, \\[1mm]
-a\,, & \mbox{ if } j = 2,5\,, \\[1mm]
0\,, & \mbox{ if } j = 3\,.
\end{array}\right.
\end{equation}
Note that the control variable $u(t)$ is a piecewise constant function, which takes here the sequence of values $\{a, -a, 0, a, -a\}$.  After evaluating the state variables in \eqref{theta_arc}--\eqref{y_arc} at the switching times and carrying out algebraic manipulations, one can equivalently rewrite Problem~(Pc) as follows.
\[
\mbox{(Ps)}\left\{\begin{array}{rl}
\min &\ \ds t_f = \sum_{j=1}^5 \,\xi_j
   \\[4mm]
\mbox{s.t.} &\ds\ x_0 - x_f + \frac{1}{a}\left(-\sin\theta_0 + 2\,\sin\theta_1 - 2\,\sin\theta_2 + 2\,\sin\theta_4 - \sin\theta_f \right) + \xi_3\,\cos\theta_2 = 0\,, \\[3mm] 
  &\ds\ y_0 - y_f + \frac{1}{a}\left(\cos\theta_0 - 2\,\cos\theta_1 + 2\,\cos\theta_2 - 2\,\cos\theta_4 + \cos\theta_f \right) + \xi_3\,\sin\theta_2 = 0\,, \\[3mm]
  &\ds\ \sin\theta_f = \sin\theta_5\,,\ \ \cos\theta_f = \cos\theta_5\,,  \\[2mm]
  &\ds\ \xi_j \ge 0\,,\ \ \mbox{ for } j = 1,\ldots,5\,,
\end{array}\right.
\]
where
\begin{equation}  \label{thetas}
\theta_1 = \theta_0 + a\,\xi_1\,,\qquad
\theta_2 = \theta_1 - a\,\xi_2\,,\qquad
\theta_4 = \theta_2 + a\,\xi_4\,,\qquad
\theta_5 = \theta_4 - a\,\xi_5\,.
\end{equation}
Substitution of $\theta_1$, $\theta_2$, $\theta_4$ and $\theta_5$ in \eqref{thetas} into Problem~(Ps) yields a finite-dimensional nonlinear optimization problem in just five variables, $\xi_j$,  $j = 1,\ldots,5$.

\begin{remark}  \label{rem:angle}  \rm
With the constraints $\sin\theta_f = \sin\theta_5$ and $\cos\theta_f = \cos\theta_5$, we make sure that we satisfy the slope condition at the terminal point.  Otherwise, the constraint $\theta_f = \theta_5$ is stronger, and imposing it might result in missing some of the feasible solutions---for example, the optimal solution itself given in Figure~\ref{fig:MD_2pts_2}(a) for the problem in Example~2.
\endproof
\end{remark}

\begin{remark}  \label{rem:global}  \rm
Problem~(Ps) can be solved by standard optimization methods and software, for example, Algencan~\cite{Andreani2007,BirMar2014}, which implements augmented Lagrangian techniques, or Ipopt~\cite{WacBie2006}, which implements an interior point method, or SNOPT~\cite{GilMurSau2005}, which implements a sequential quadratic programming algorithm.  For general nonconvex optimization problems like Problem~(Ps), what one can hope for, by using these software, is to get (at best) a locally optimal solution.  

Recall that, by Theorem~\ref{stationarity}, there will be as many stationary solutions as the number of feasible solutions of the listed types.  Suppose that a solution of Problem~(Pc) is found by software, with locally optimal length $t_f = L_0$.  Since this solution is likely not to be a global solution, Problem~(Ps) can be modified by adding the constraint $\sum_{j=1}^5 \,\xi_j \le L_0 + \varepsilon$, with some small $\varepsilon > 0$, and then Problem~(Ps) numerically solved again, which might provide a curve with a length $L_1 < L_0$.  Suppose that a shorter curve with length $L_1$ is found.  Then the same procedure can be repeated in, say, Step $i+1$, $i = 1,2,\ldots$, by adding the constraint $\sum_{j=1}^5 \,\xi_j \le L_i + \varepsilon$, until it is no longer possible to find a curve with a length shorter than $L_i$.

At this point, one might rightfully argue that, since the number of variables in Problem~(Ps) is only five, rather than using the ``global procedure'' described in the previous paragraph, one can employ some general global optimization method and associated software which might serve as a remedy in finding a shortest curve solution.  However, as we pointed earlier, Problem~(Ps) is set up here as a prototype/building block for the Markov-Dubins interpolating curves, in which case, the companion of Problem~(Ps) will in general have a significantly larger number of variables, prohibiting employment of a global solver.
\endproof
\end{remark}

\subsection{Numerical experiments}

In this section, we present numerical experiments by solving Problem~(Ps) under various sets of data and use the globalization strategy described in Remark~\ref{rem:global}, to construct Markov-Dubins curves and other stationary curves.  For solving Problem~(Ps), we use Ipopt, version 3.12, a popular optimization software based on an interior point method~\cite{WacBie2006}.  AMPL~\cite{AMPL} is used as an optimization modelling language, which employs Ipopt as a solver.  All Ipopt tolerances are set at $10^{-15}$, with MA77 chosen as the linear solver. \\

\noindent
{\bf Example 1} \\[2mm]
Figure~\ref{fig:MD_2pts}(a) depicts a CSC type, or more specifically an LSR type, Markov Dubins path of maximum curvature of 3 units (or, minimum turning radius of 1/3 units) between the initial and terminal points $(0,0)$ and $(1,1)$, obtained by solving Problem~(Ps), a number of times by adding a constraint, which imposes an upper bound on the length of the curve.  The initial and terminal orientation angles are specified as $-60^\circ$ and $-30^\circ$, respectively.

\begin{figure}[t]
\begin{minipage}{80mm}
\begin{center}
\psfrag{x}{\small $x$}
\psfrag{y}{\small $y$}
\includegraphics[width=80mm]{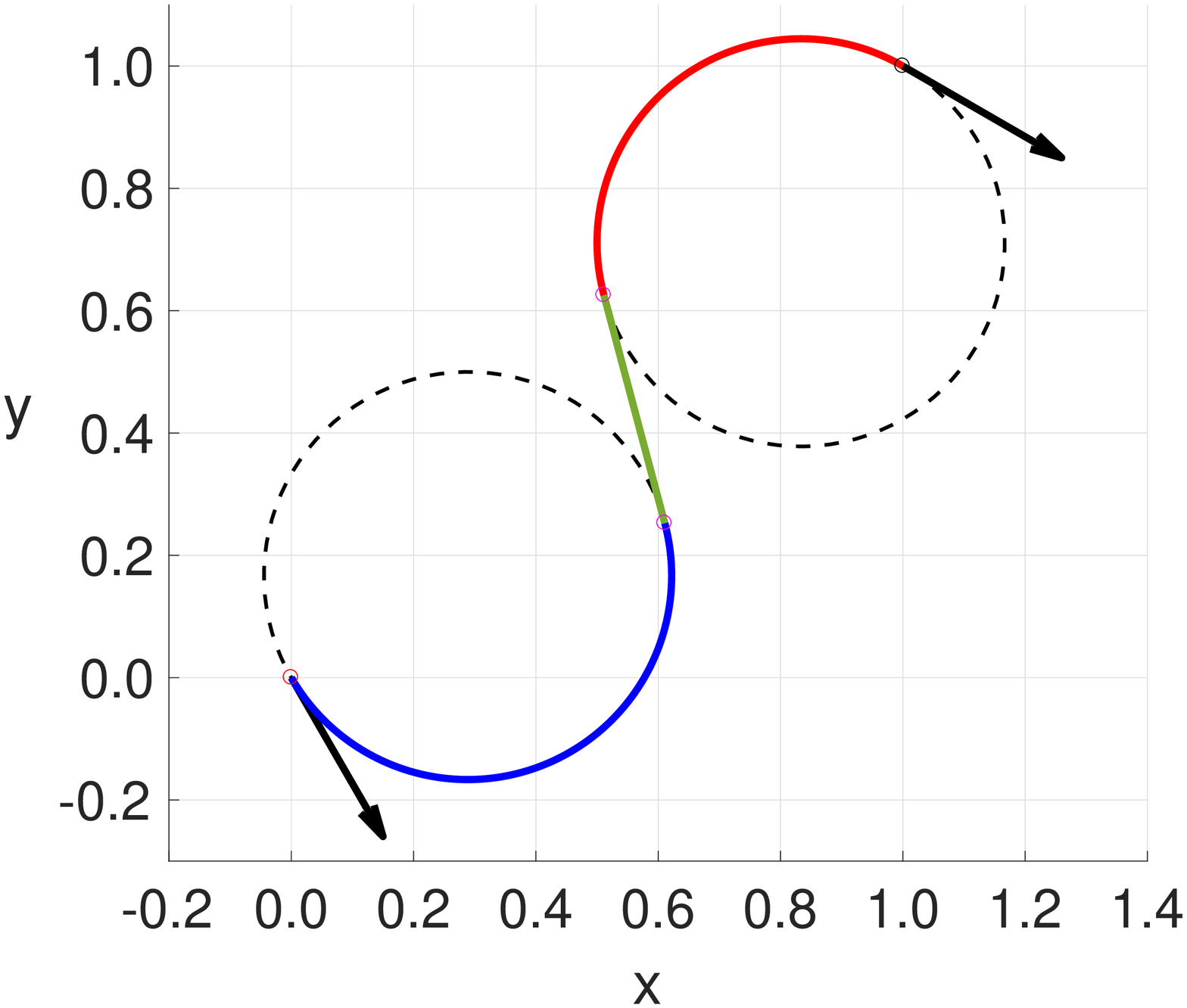} \\[3mm]
{\small (a) Markov-Dubins curve}
\end{center}
\end{minipage}
\begin{minipage}{80mm}
\begin{center}
\psfrag{t}{\small $t$}
\psfrag{lambda3}{\small\hspace*{0mm} $\lambda_3(t)$}
\includegraphics[width=80mm]{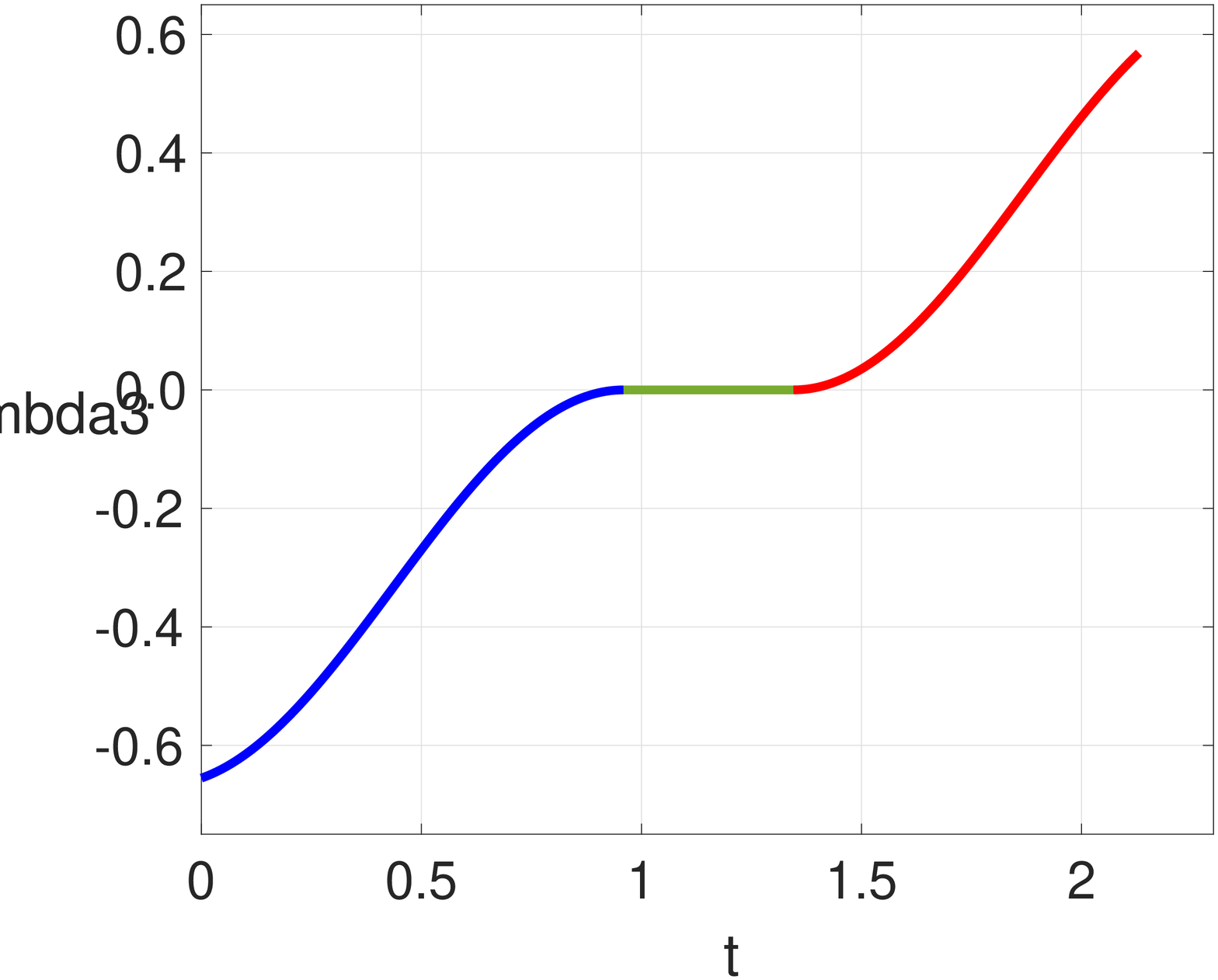} \\[3mm]
{\small (b) Switching function, $\lambda_3$}
\end{center}
\end{minipage}
\\[5mm]
\begin{minipage}{80mm}
\begin{center}
\psfrag{lambda3}{\small $\lambda_3$}
\psfrag{lambda3dot}{\small $\dot{\lambda}_3$}
\includegraphics[width=80mm]{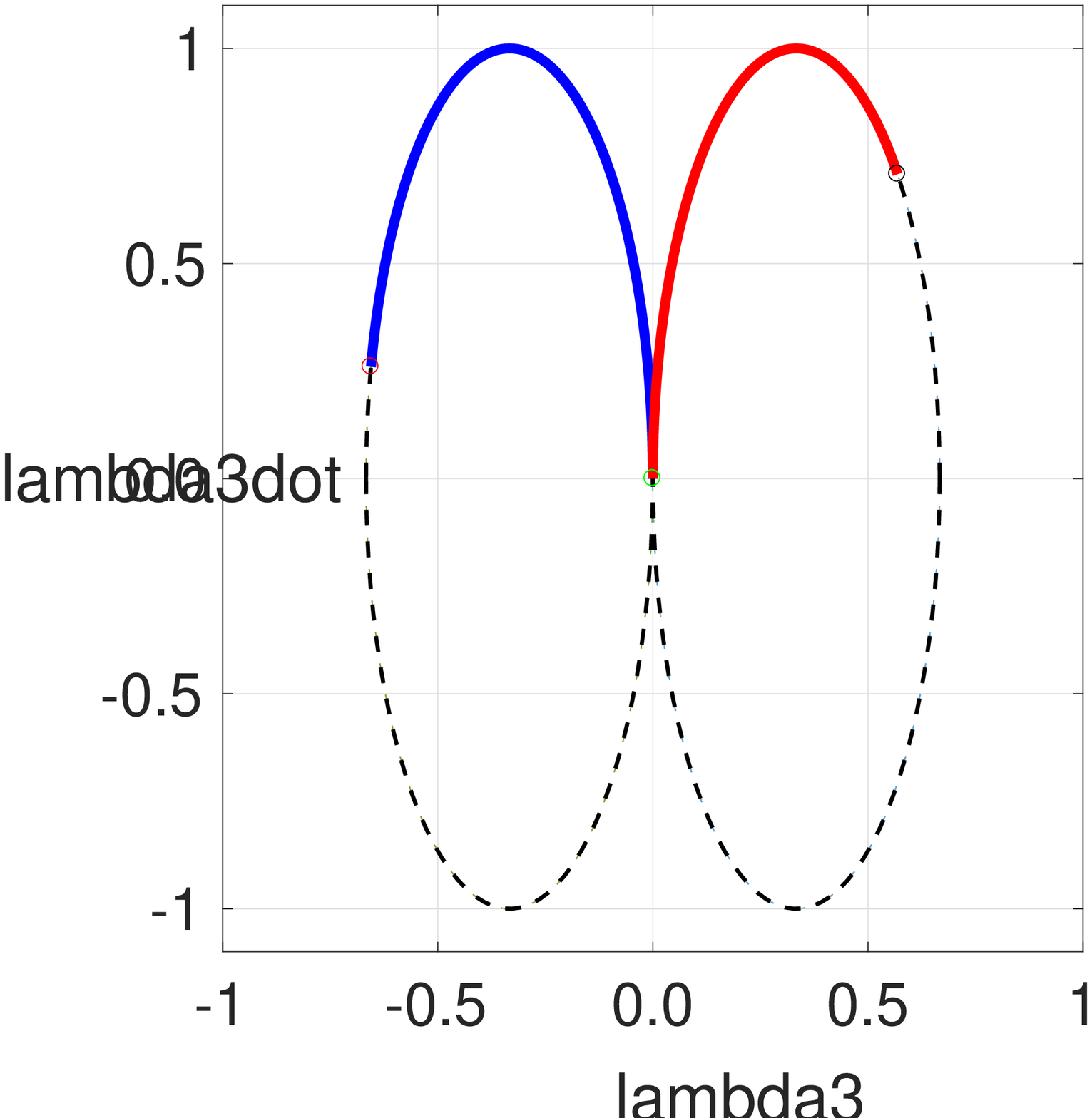} \\[3mm]
{\small (c) Phase diagram for $\lambda_3$}
\end{center}
\end{minipage}
\begin{minipage}{80mm}
\begin{center}
\psfrag{x}{\small $x$}
\psfrag{y}{\small $y$}
\includegraphics[width=80mm]{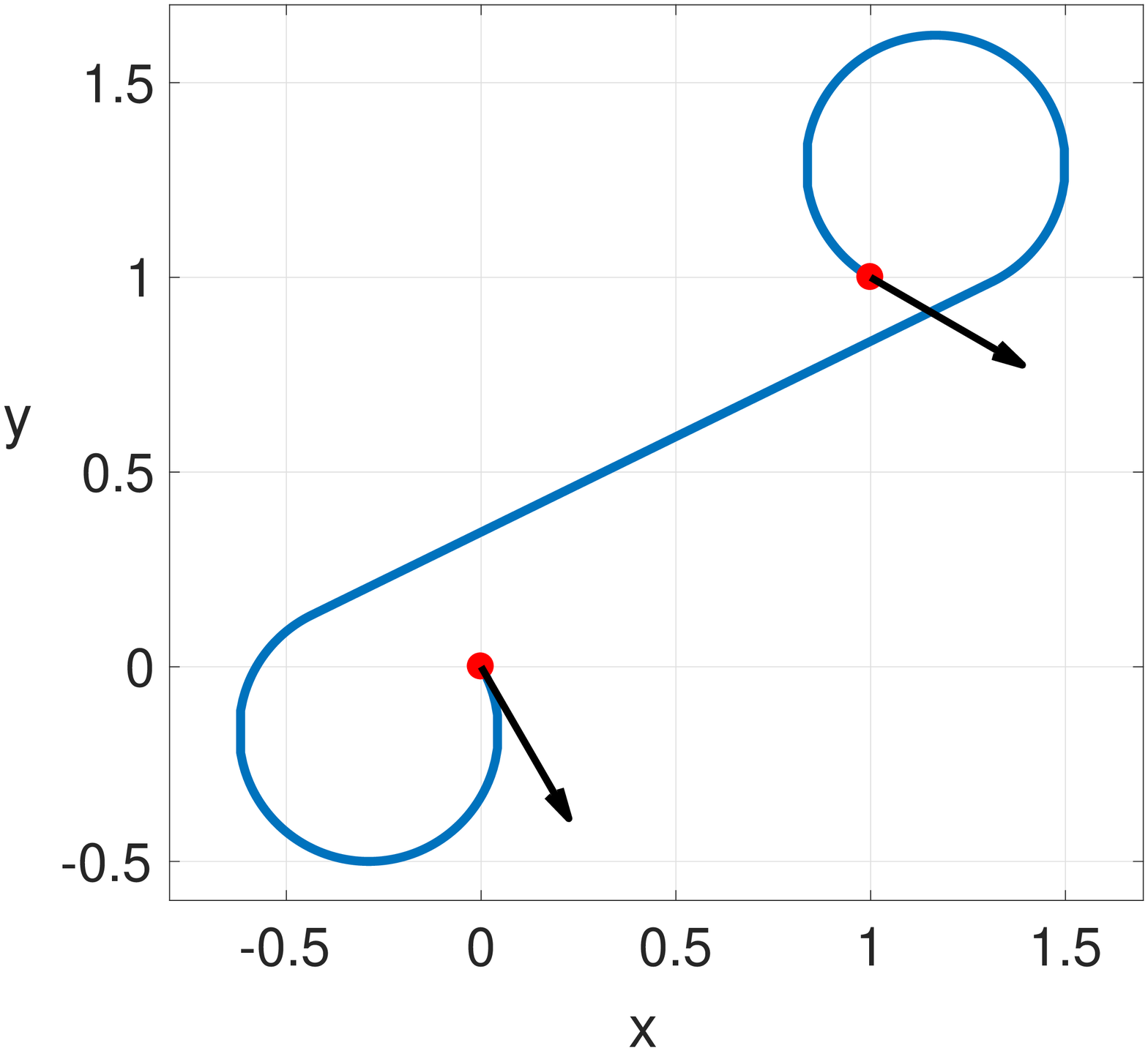} \\[3mm]
{\small (d) A stationary solution curve}
\end{center}
\end{minipage}
\
\caption{\small\sf (a) Markov-Dubins curve from $(0,0,-\pi/3)$ to $(1,1,-\pi/6)$ with a maximum curvature of 3 units, (b)--(c) the graphs for the adjoint variable~$\lambda_3$ and (d) a longer stationary curve between the same points.}
\label{fig:MD_2pts}
\end{figure}

The solution to Problem~(Ps), by following the procedure in the second paragraph of Remark~\ref{rem:global}, results in a Markov-Dubins curve of type $L_{\xi_1}S_{\xi_3}R_{\xi_5}$ (with ${\xi_2} = {\xi_4} = 0$), where
\[
\xi_1 = 0.95958462\,,\quad
\xi_3 = 0.38582465\,,\quad
\xi_5 = 0.78505169\,.
\]
Re-define the switching times, for simplicity, as $t_1 = \xi_1$, $t_2 = \xi_1 + \xi_3$ and $t_f = \xi_1 + \xi_3 + \xi_5$, which are different from the definition given in \eqref{arc_durations} but more intuitive for our practical purposes.  Then
\[
t_1 = 0.95958462\,,\quad
t_2 = 1.34540927\,,\quad
t_f = 2.13046097\,.
\]
The graph of the switching function, $\lambda_3(t)$, in Figure~\ref{fig:MD_2pts}(b), as well as the phase plane trajectory in Figure~\ref{fig:MD_2pts}(c), have been drawn simply by using the expression~\eqref{lambda3_CSC}.  These graphs are included here for illustration purposes; otherwise, we already know by Theorem~\ref{stationarity} that any numerical solution to Problem~(Ps) of certain types satisfies the Pontryagin maximum principle.  

In Figure~\ref{fig:MD_2pts}(d), we provide a stationary solution of Problem~(Ps), which is of type $R_{\xi_2}S_{\xi_3}L_{\xi_4}$ (with ${\xi_1} = {\xi_5} = 0$), where
\[
\xi_2 = 1.5934841453\,,\quad
\xi_3 = 1.9472018572\,,\quad
\xi_4 = 1.7680170705\,,
\]
and so, with $t_1 = \xi_2$, $t_2 = \xi_2 + \xi_3$ and $t_f = \xi_2 + \xi_3 + \xi_4$,
\[
t_1 = 1.5934841453\,,\quad
t_2 = 3.540686003\,,\quad
t_f = 5.308703073\,.
\]
A comparison of the lengths $t_f$ of the two curves tells us that the stationary curve provided in Figure~\ref{fig:MD_2pts}(d) is more than twice longer than the Markov-Dubins curve in Figure~\ref{fig:MD_2pts}(a), although it does not appear so since Figure~\ref{fig:MD_2pts}(d) is presented on a different scale.  It should be noted that, in addition to the stationary solution curve in Figure~\ref{fig:MD_2pts}(d), there are two further stationary solution curves, which are of type $RSR$ with $t_f = 3.34456289$, and type $LSL$ with $t_f = 3.69362874$, respectively, either of which can be found by solving Problem~(Ps). Each of these solutions is normal, i.e., $\lambda_0 \neq 0$, by the results proved before.\\

\noindent
{\bf Example 2} \\[2mm]
The only difference of the problem in the present example from the one in Example~1 is that as the terminal point we take $(0.4,0.4)$ instead of $(1,1)$.  Figure~\ref{fig:MD_2pts_2} depicts six of the seven stationary solutions to Problem~(Pc), obtained by solving Problem~(Ps), the first one of which, shown in Figure~\ref{fig:MD_2pts_2}(a), is the Markov-Dubins curve.  The seventh stationary solution, which is not shown in the figure, is of type $RSL$ with $t_f = 4.54008162$.   As can be seen, there are two stationary solutions each of types $RLR$ (Figure\ref{fig:MD_2pts_2}(b) and (e)) and $LRL$ (Figure\ref{fig:MD_2pts_2}(d) and (f)). There does not seem to be a stationary (or feasible) solution of type $LSR$ or of any other type which concatenates less than three subarcs.

As in Example~1, the adjoint variable, or the switching function, $\lambda_3(t)$, can be easily computed, for each of the solutions in Figure~\ref{fig:MD_2pts_2}(a)--(f).  For the solutions in Figure~\ref{fig:MD_2pts_2}(a)--(b) and (e), the formula in~\eqref{lambda3_CSC}, and for the solutions in Figure~\ref{fig:MD_2pts_2}(c)--(d), the formula in~\eqref{lambda3_CCC} can be used.  Note that by the results obtained before, each of these solutions is normal, i.e., $\lambda_0 \neq 0$. \\

\afterpage{\clearpage}
\begin{figure}[t]
\begin{minipage}{80mm}
\begin{center}
\psfrag{x}{\small $x$}
\psfrag{y}{\small $y$}
\includegraphics[width=80mm]{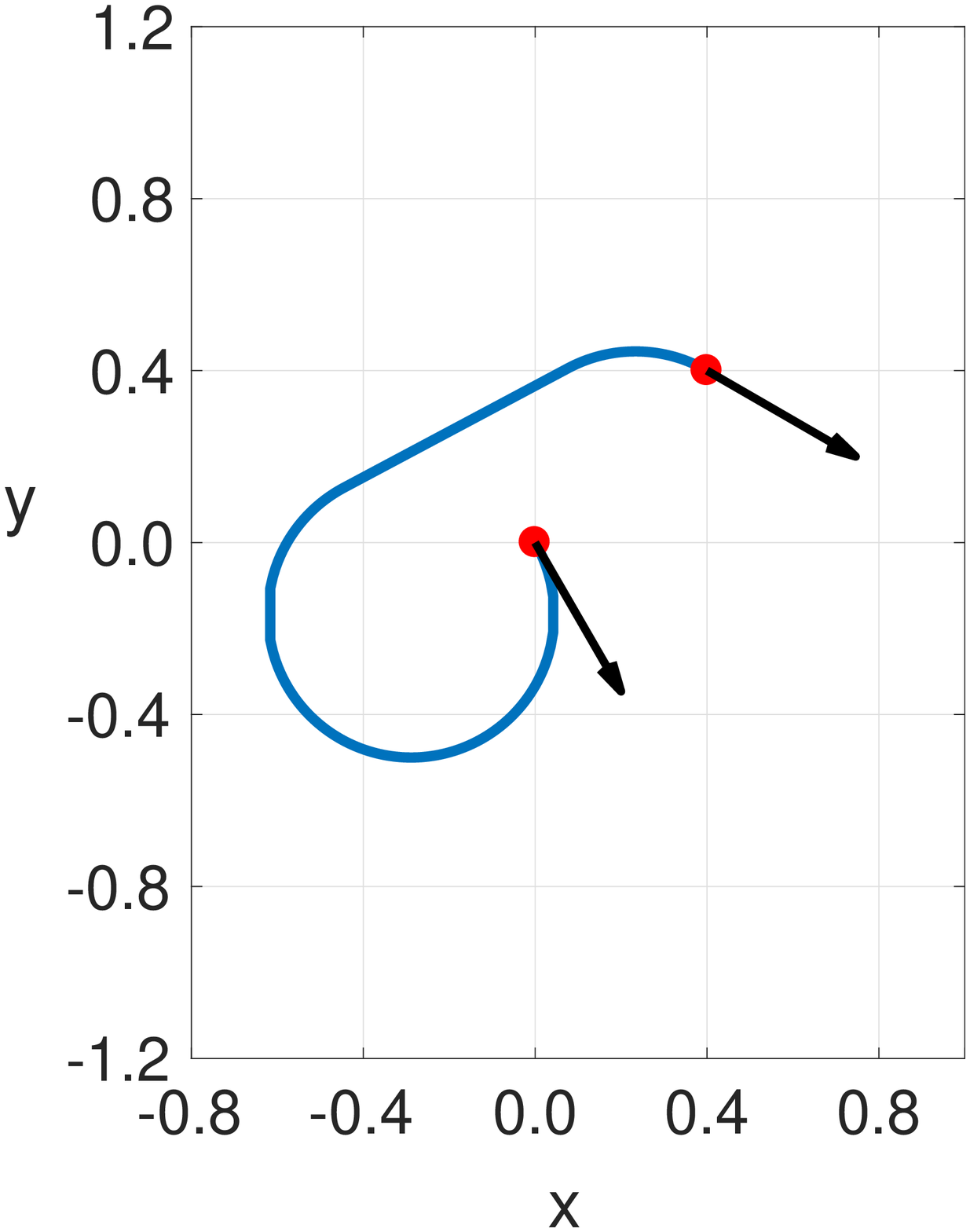} \\[3mm]
{\small (a) Type $RSR$;\ \ $t_f = 2.51127753$}
\end{center}
\end{minipage}
\begin{minipage}{80mm}
\begin{center}
\psfrag{x}{\small $x$}
\psfrag{y}{\small $y$}
\includegraphics[width=60mm]{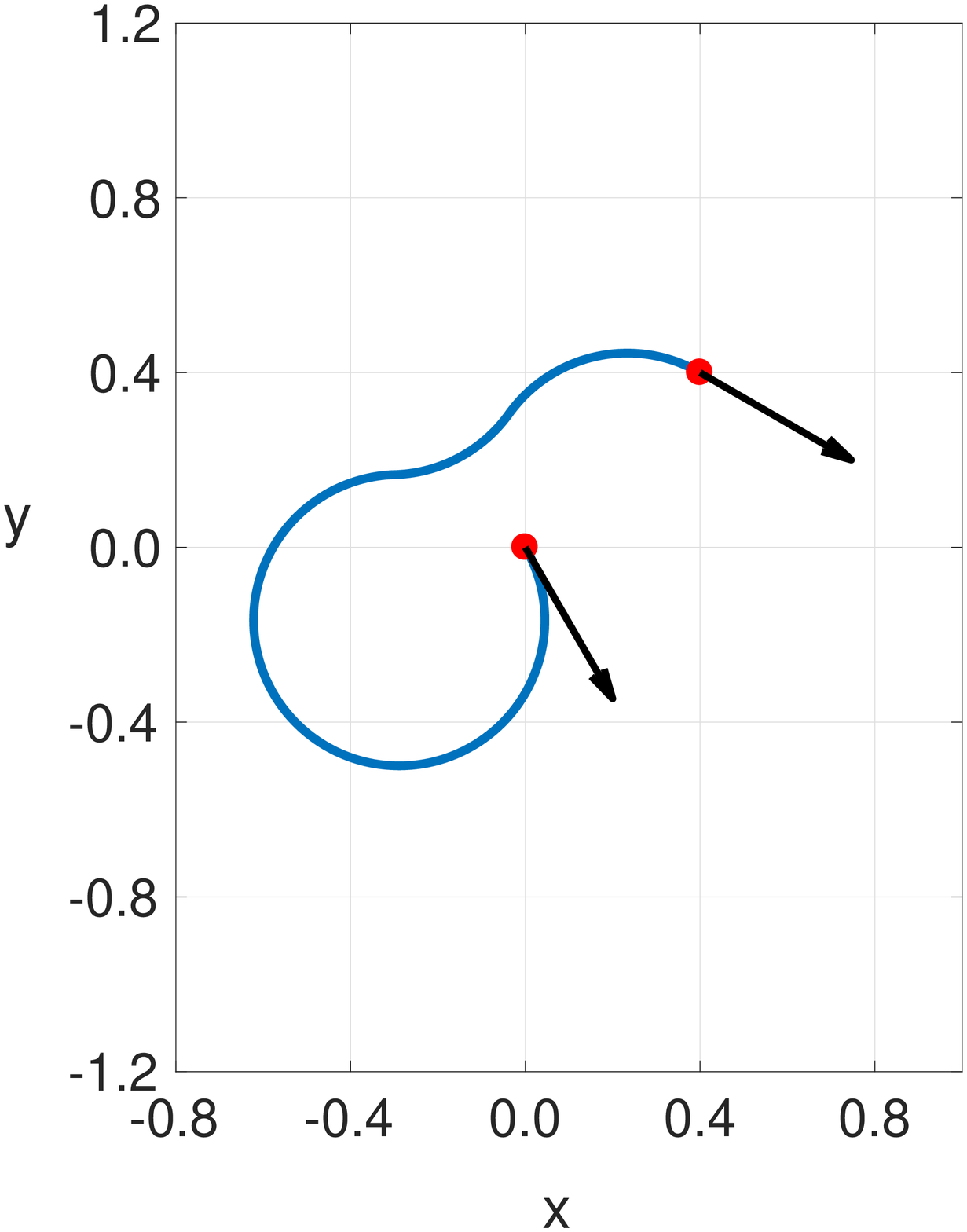} \\[3mm]
{\small (b) Type $RLR$;\ \ $t_f = 2.53262033$}
\end{center}
\end{minipage}
\\[5mm]
\begin{minipage}{80mm}
\begin{center}
\psfrag{x}{\small $x$}
\psfrag{y}{\small $y$}
\includegraphics[width=80mm]{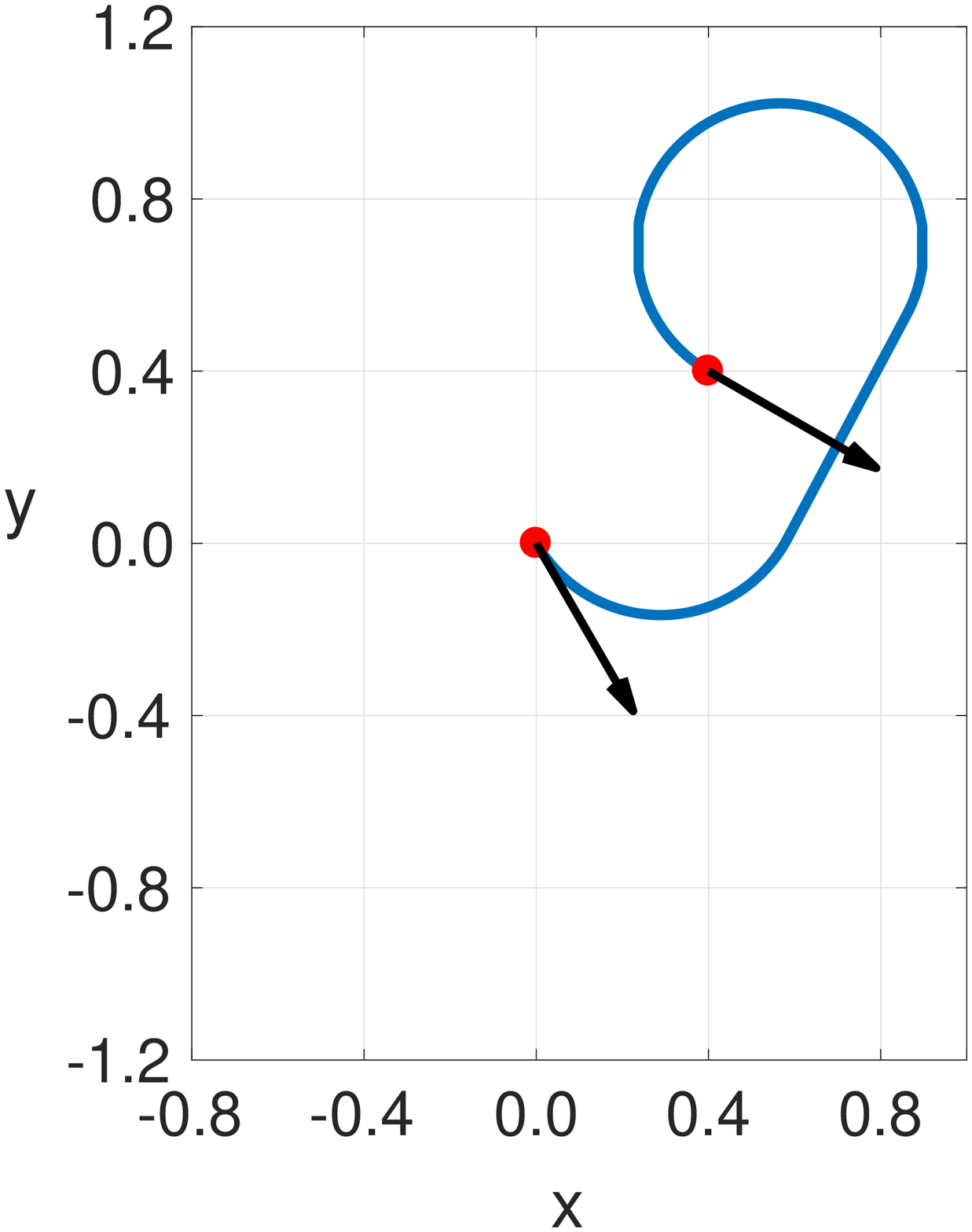} \\[3mm]
{\small (c) Type $LSL$;\ \ $t_f = 2.86034339$}
\end{center}
\end{minipage}
\begin{minipage}{80mm}
\begin{center}
\psfrag{x}{\small $x$}
\psfrag{y}{\small $y$}
\includegraphics[width=60mm]{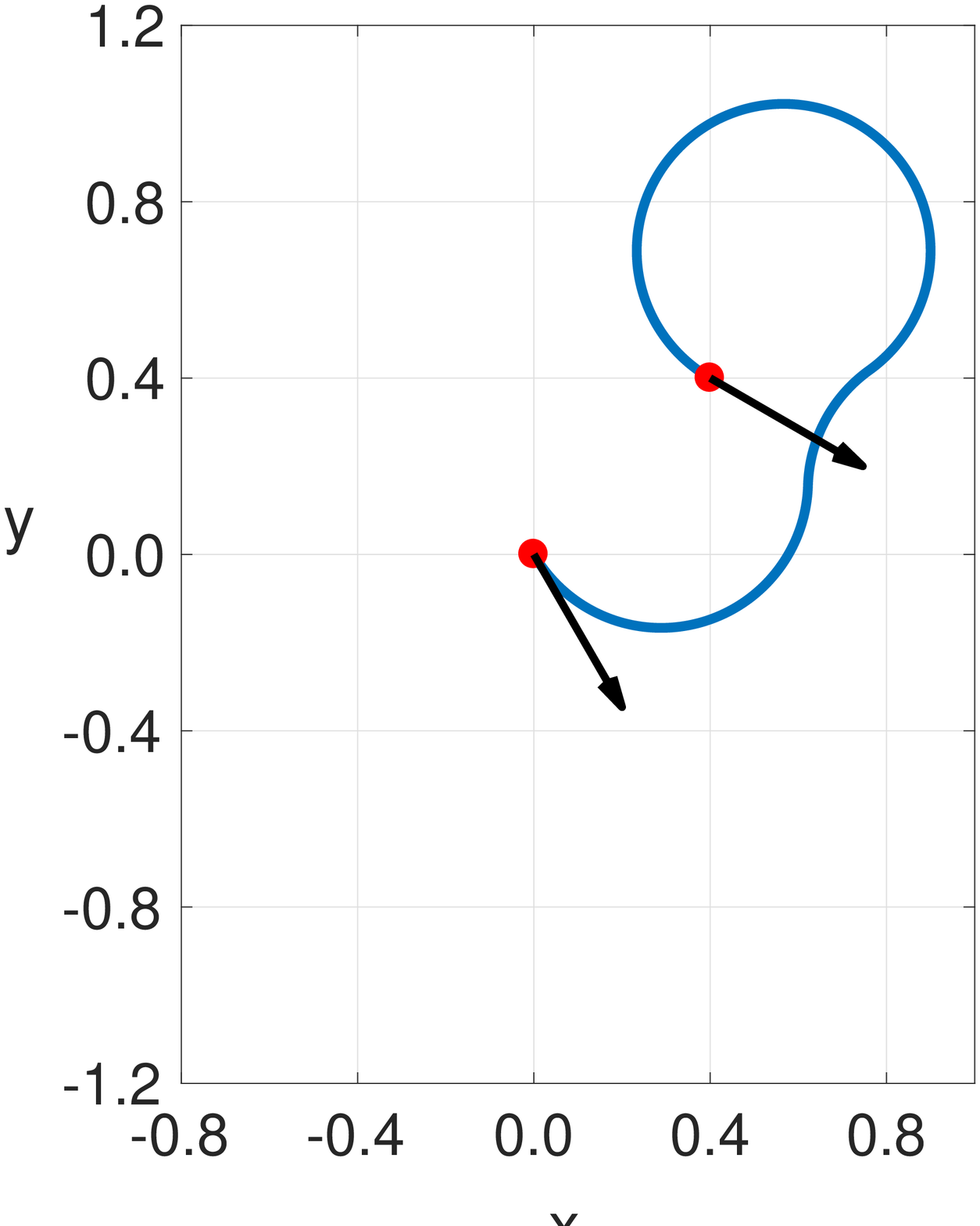} \\[3mm]
{\small (d) Type $LRL$;\ \ $t_f = 2.88168618$}
\end{center}
\end{minipage}
\\[5mm]
\begin{minipage}{80mm}
\begin{center}
\psfrag{x}{\small $x$}
\psfrag{y}{\small $y$}
\includegraphics[width=80mm]{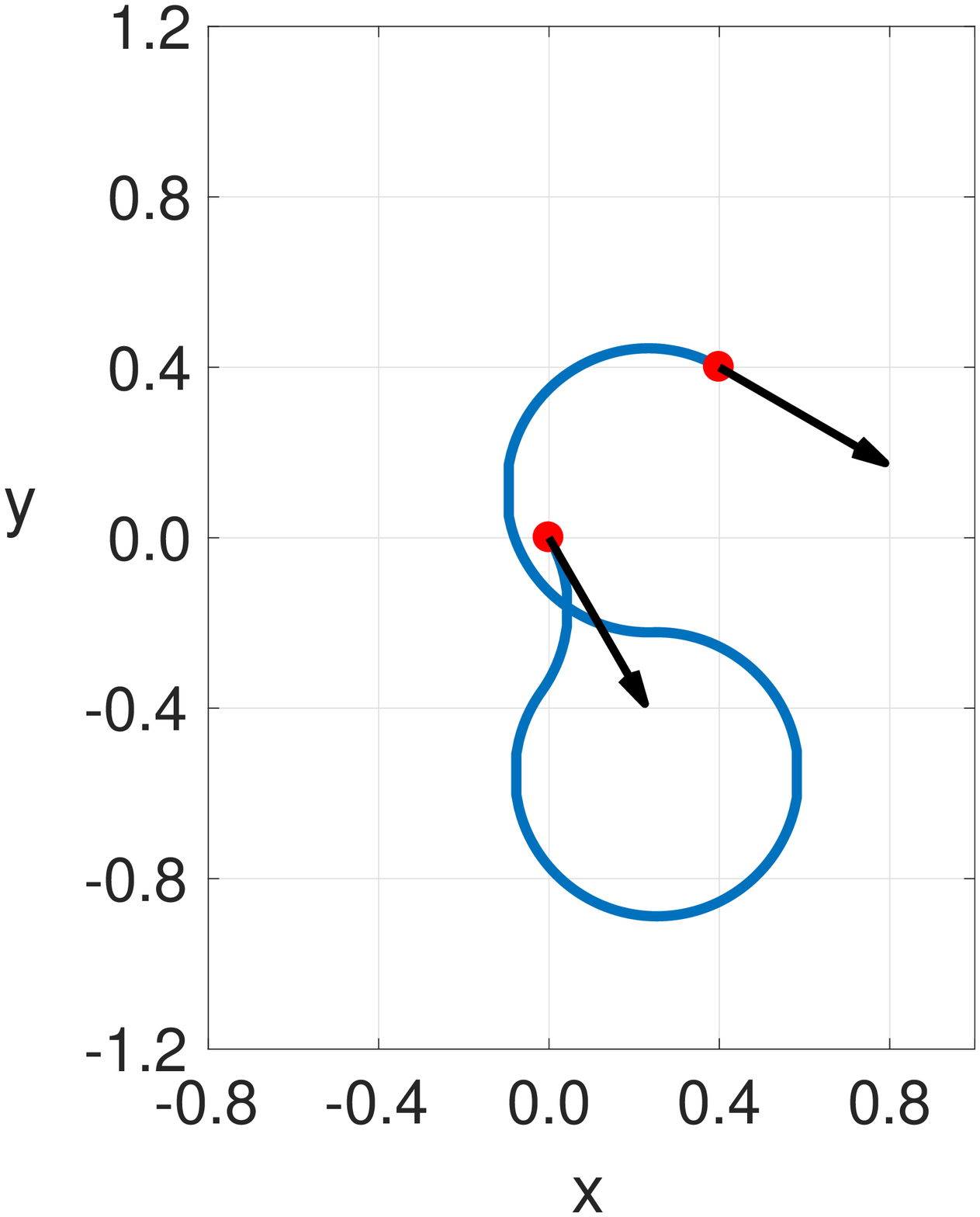} \\[3mm]
{\small (e) Type $RLR$;\ \ $t_f = 3.40149913$}
\end{center}
\end{minipage}
\begin{minipage}{80mm}
\begin{center}
\psfrag{x}{\small $x$}
\psfrag{y}{\small $y$}
\includegraphics[width=80mm]{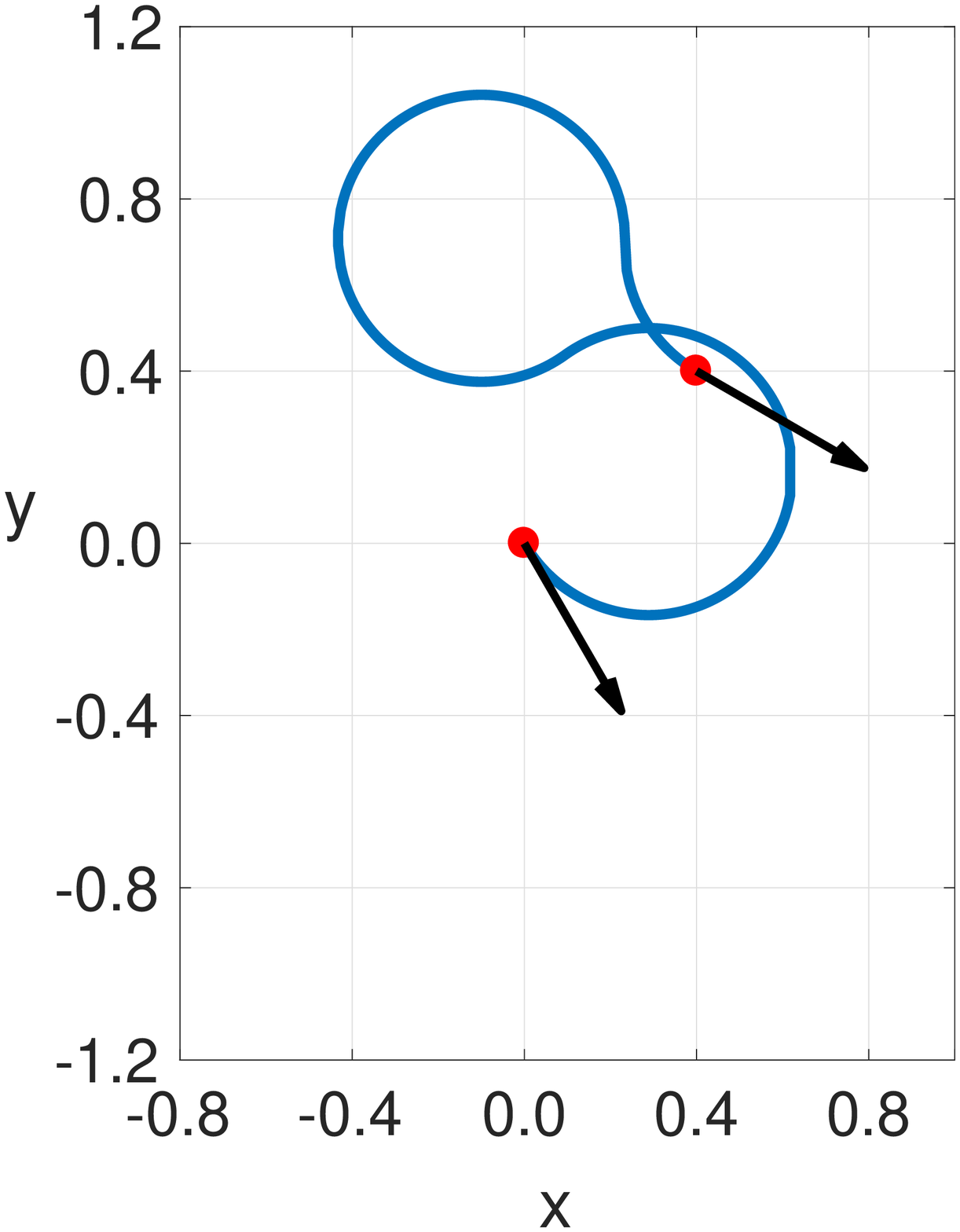} \\[3mm]
{\small (f) Type $LRL$;\ \ $t_f = 3.75056498$}
\end{center}
\end{minipage}
\
\caption{\small\sf (a) Markov-Dubins curve from $(0,0,-\pi/3)$ to $(0.4,0.4,-\pi/6)$ with a maximum curvature of 3 units and (b)--(f) some other stationary curves between the same oriented points.}
\label{fig:MD_2pts_2}
\end{figure}
\clearpage

\noindent
{\bf Example 3 (An Abnormal Case)} \\[2mm]
Here, we provide an example to the solution of an abnormal problem.  Suppose that we want to find a Markov Dubins path of maximum curvature of 1 unit (or, minimum turning radius of 1 unit) between the initial and terminal points $(0,0)$ and $(4,4)$, where the initial and terminal orientation angles are both $-90^\circ$.

\begin{figure}[t]
\begin{minipage}{80mm}
\begin{center}
\psfrag{x}{\small $x$}
\psfrag{y}{\small $y$}
\includegraphics[width=80mm]{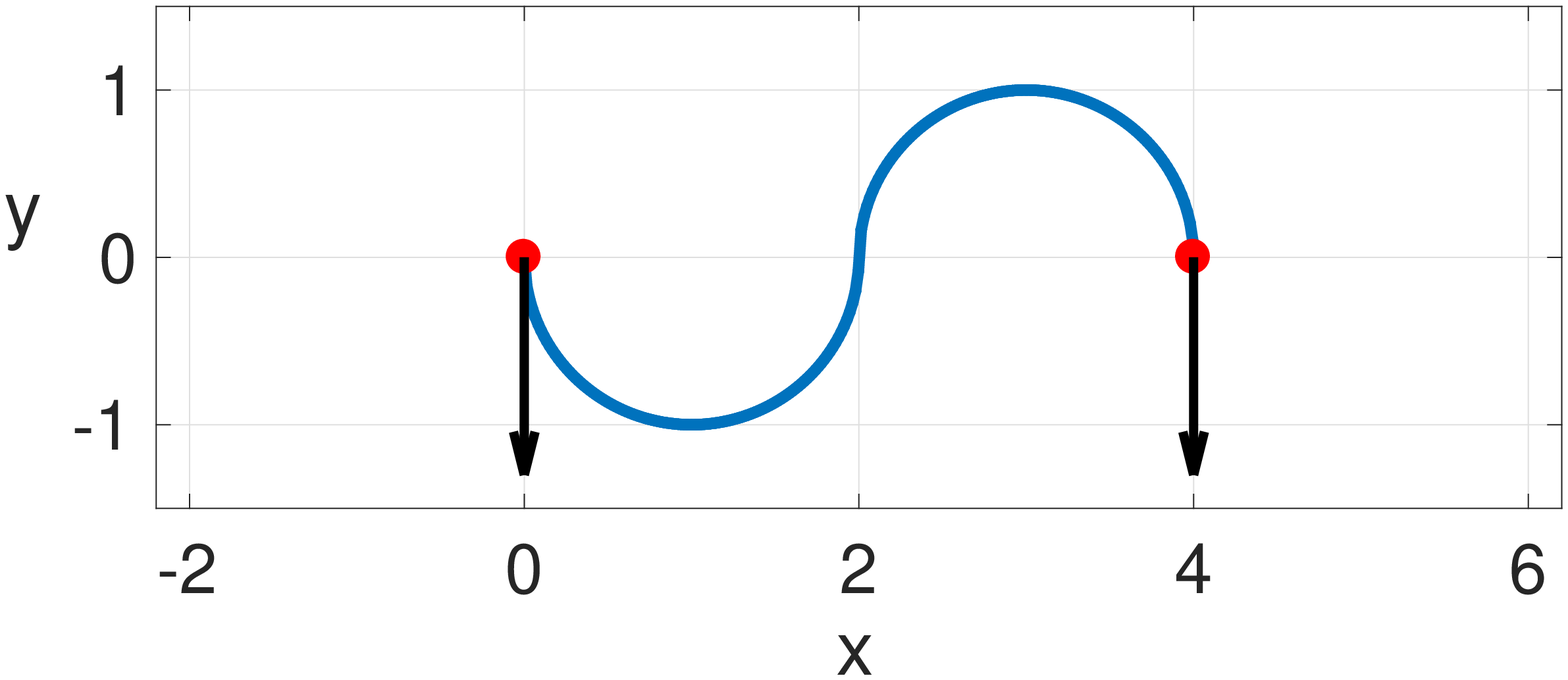} \\[3mm]
    (a) {\small Type $LR$;\ \ $t_f = 2\,\pi$}
\end{center}
\end{minipage}
\begin{minipage}{80mm}
\begin{center}
\psfrag{x}{\small $x$}
\psfrag{y}{\small $y$}
\includegraphics[width=80mm]{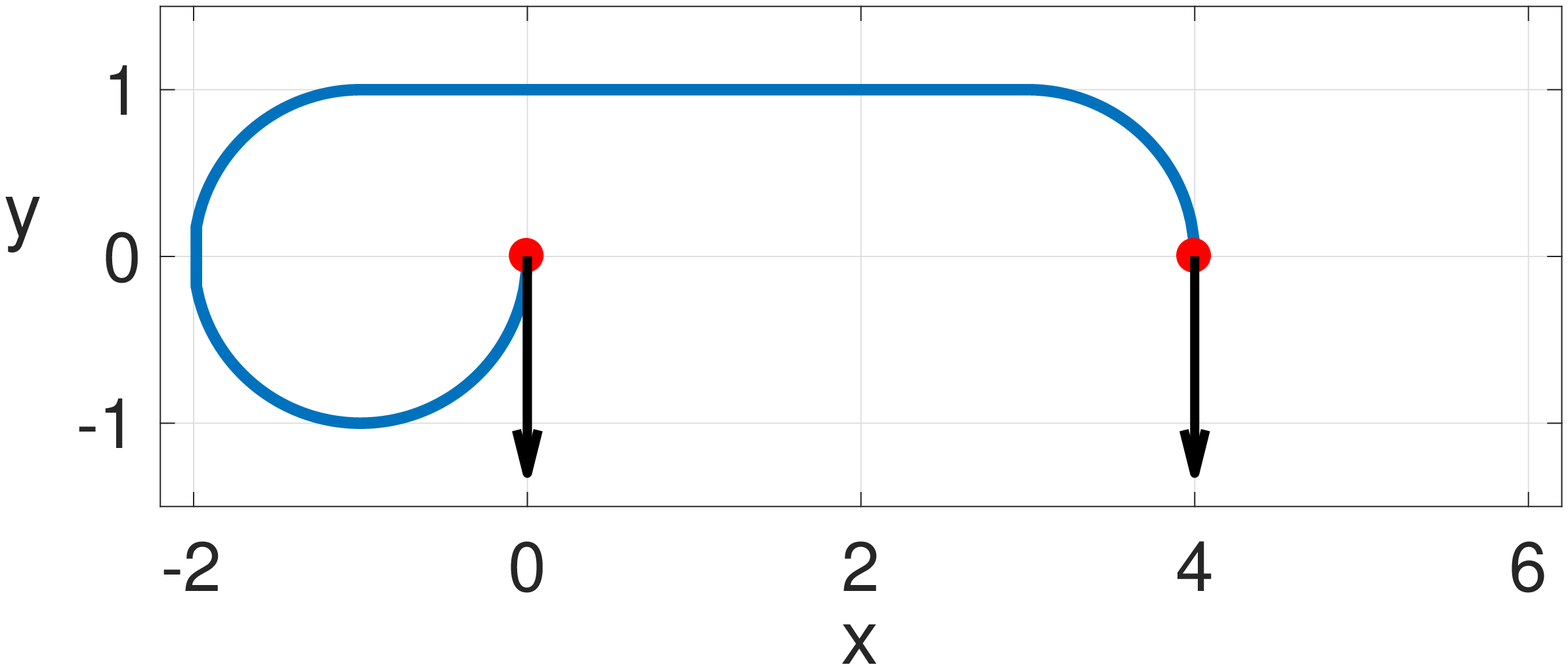} \\[3mm]
    (b) {\small Type $RSR$;\ \ $t_f = 2\,\pi + 4$}
\end{center}
\end{minipage}
\\[5mm]
\begin{minipage}{80mm}
\begin{center}
\psfrag{x}{\small $x$}
\psfrag{y}{\small $y$}
\includegraphics[width=80mm]{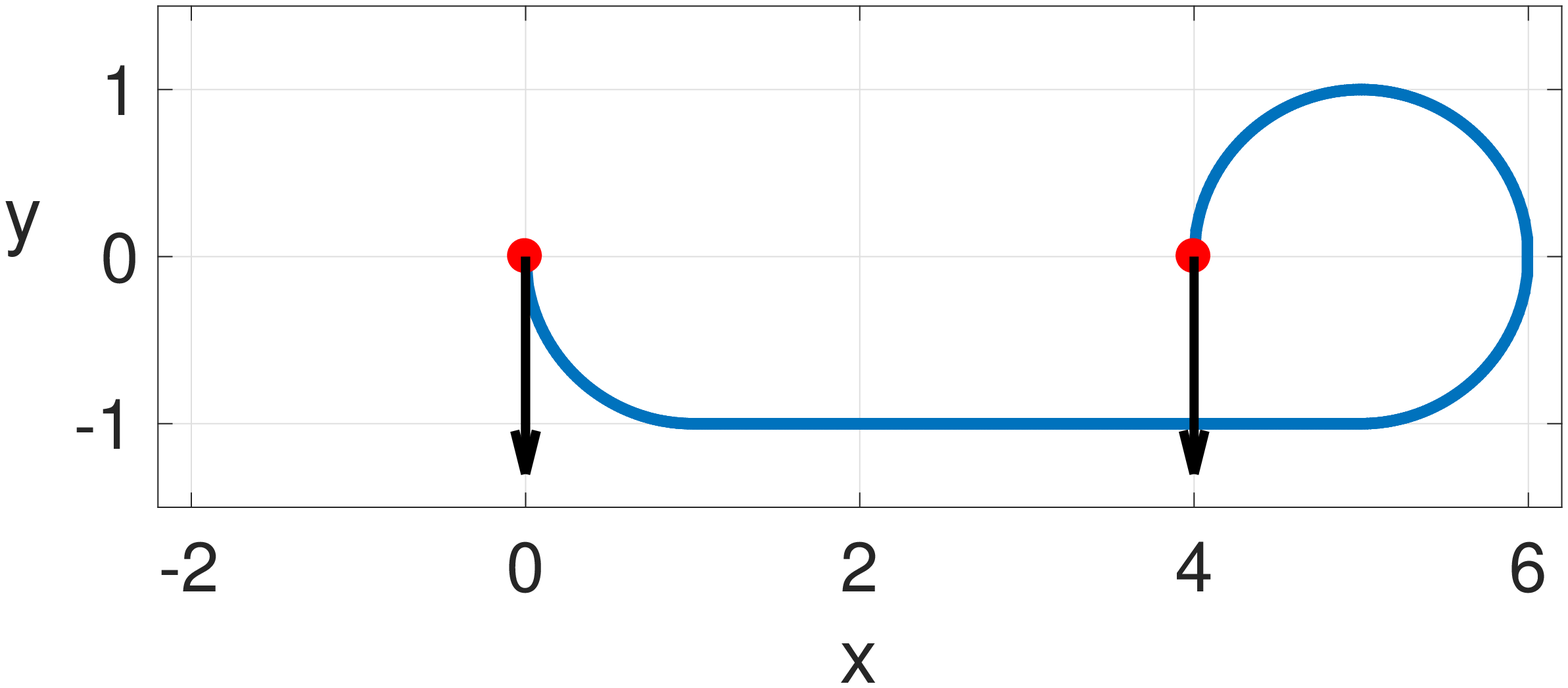} \\[3mm]
(c) {\small Type $LSL$;\ \ $t_f = 2\,\pi + 4$}
\end{center}
\end{minipage}
\begin{minipage}{80mm}
\begin{center}
\psfrag{x}{\small $x$}
\psfrag{y}{\small $y$}
\includegraphics[width=80mm]{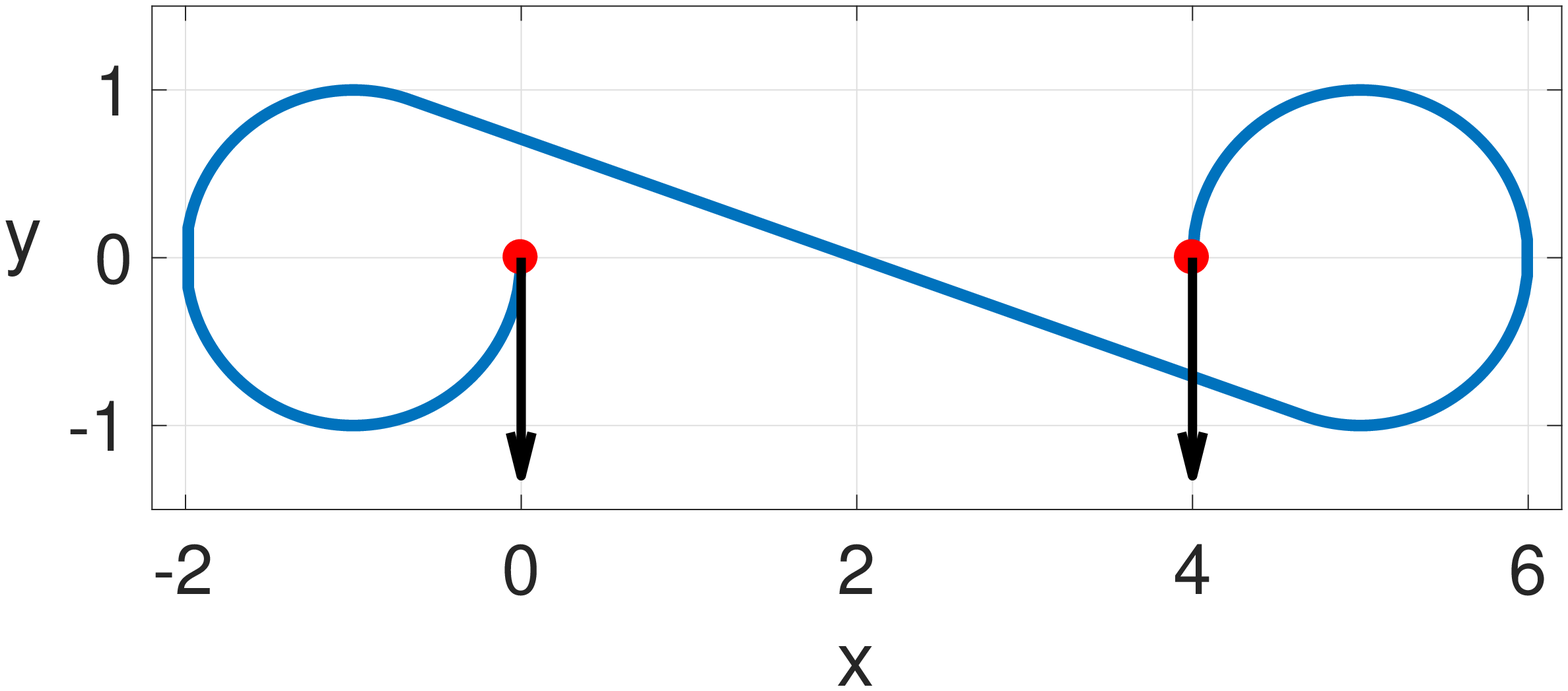} \\[3mm]
\hspace*{-6mm}(d) T{\small ype $RSL$;\ \ $t_f = 2[2\pi - \cos^{-1}(1/3) + 2\sqrt{2}]ñ$}
\end{center}
\end{minipage}
\
\caption{\small\sf (a) {\em Abnormal Markov-Dubins curve} from $(0,0,-\pi/2)$ to $(4,0,-\pi/2)$ with a maximum curvature of 3 units and (b)--(d) the other stationary curves between the same oriented points.}
\label{fig:MD_2pts_3}
\end{figure}

Figure~\ref{fig:MD_2pts_3} depicts all of the stationary solution curves for the given problem.  The one shown in (a) is a Markov-Dubins curve, for which $\lambda_0 = 0$, and so the curve is abnormal:  Note that the trajectory in Figure~\ref{fig:MD_2pts_3} corresponds to exactly one cycle along an ellipse in Figure~\ref{abnormal_phase} and that switching occurs when $\theta_1 = \pi/2$.  This is a situation that can also happen with $\lambda_0\neq0$, by Corollary~\ref{normal_abnormal_2}, and this can be verified by the phase portrait in Figure~\ref{phase}.   The adjoint variable, or the switching function, $\lambda_3(t)$, for the abnormal Markov-Dubins curve here can be found by using the formula in~\eqref{lambda3_CC_abnormal}.  By Corollary~\ref{normal_abnormal_2}, the abnormal Markov-Dubins curve here is also normal, and the associated switching function, $\lambda_3(t)$ can be obtained by using the formula in~\eqref{lambda3_CC}.

The adjoint variable $\lambda_3(t)$ for the stationary solutions depicted in Figure~\ref{fig:MD_2pts_2}(b)--(d), on the other hand, can be found by using the formula in~\eqref{lambda3_CSC}.  It should be noted that these solutions in Figure~\ref{fig:MD_2pts_2}(b)--(d) are only normal, for which $\lambda_0\neq0$.

\section{Conclusion and Further Work}

We have presented a study of the Markov-Dubins problem by employing optimal control theory, and reproduced Dubins' result (Theorem~\ref{Dubins}) which classifies the types of the shortest curves of bounded-curvature between two oriented points.  We have shown that the optimal control formulation of the problem can be abnormal, as well as normal (Lemmas~\ref{rho_nonsing} and~\ref{abnormal} and Corollaries~\ref{normal_abnormal_1} and~\ref{normal_abnormal_2}).  We characterized these solutions. We have also shown that feasible solutions of the types an optimal solution is required to be of are stationary (Theorem~\ref{stationarity}).  We presented a numerical method based on switching time parameterization, and applied it to two example problems to illustrate the method as well as the results on the normal Markov-Dubins paths.  With an additional example, we illustrated the presence of an abnormal Markov-Dubins path.

The results, as well as the numerical method, presented in this study constitute building blocks for a study of a generalization to {\em Markov-Dubins interpolating curves}, where the curve is required to pass through a number of additional intermediate points.  Work in this direction is in progress, incorporating additional techniques from optimal control theory available in \cite{AugMau2000,ClaVin1989,KayNoa2013}.

Theorem~\ref{stationarity} proves stationarity of feasible curves of certain types, but does not say anything about the local optimality of these curves.  The proof of Theorem~\ref{stationarity}, on the other hand, provides explicit expressions for the switching function in both the normal and abnormal cases.  These expressions could possibly be used in verifying the second-order sufficient conditions of optimality, which can be found in \cite{MauBueKimKay2005}, for curves of type $CCC$, i.e., for bang--bang type control.  To the author's knowledge, the only result on the local optimality of a stationary curve of the Markov-Dubins problem is provided in~\cite{AroBonDmiLot2012} for a curve of type $CS$, as a case study.

Another interesting extension of the results and the numerical method in the current paper would be a study of the Markov-Dubins problem in the three-dimensional space.  Sussmann presents in~\cite{Sussmann1995}  a study of the three-dimensional Markov-Dubins problem, by using geometric optimal control theory, and finds that the optimal path can include a helicoidal arc.

In practical situations, for example in the flight trajectory planning of a UAV, constraints are often imposed because of the terrain, flight traffic and no-go areas.  Therefore, yet another interesting extension of the current study would be to situations where constraints on the space are added, giving rise to state-constrained optimal control problems, which are significantly more challenging.

\section*{Acknowledgment}
The author would like to thank Helmut Maurer for pointing to Reference~\cite{AroBonDmiLot2012}, after a talk the author had presented on the topic.

\end{document}